\documentclass{article}
\usepackage{graphicx} % Required for inserting images
\usepackage{amsmath}
\usepackage{amsfonts}
\usepackage{amsthm}
\usepackage{amssymb}
\usepackage{xcolor}
\usepackage{todonotes}

\newtheorem{thm}{Theorem}[section]

\newtheorem{lem}[thm]{Lemma}
\newtheorem{cor}[thm]{Corollary}

\theoremstyle{remark}
\newtheorem*{rmk}{Remark}
\theoremstyle{definition}
\newtheorem*{ex}{Example}

\DeclareMathOperator{\Res}{Res}
\DeclareMathOperator{\lcm}{lcm}

\def\thesistitle{Sequences related to Lehmer's problem}
\def\thesissubtitle{}
\def\thesisauthorfirst{Björn}
\def\thesisauthorsecond{Johannesson}
\def\thesissupervisorfirst{Wadim}
\def\thesissupervisorsecond{Zudilin}
\def\thesissecondreaderfirst{Riccardo}
\def\thesissecondreadersecond{Cristoferi}
\def\thesisdate{June 29, 2023}

\title{\thesistitle}
\author{\thesisauthorfirst\space\thesisauthorsecond}
\date{\thesisdate}

\begin{document}

\begin{titlepage}
	\thispagestyle{empty}
	\newcommand{\HRule}{\rule{\linewidth}{0.5mm}}
	\center
	\textsc{\Large Radboud University Nijmegen}\\[.7cm]
	\includegraphics[width=25mm]{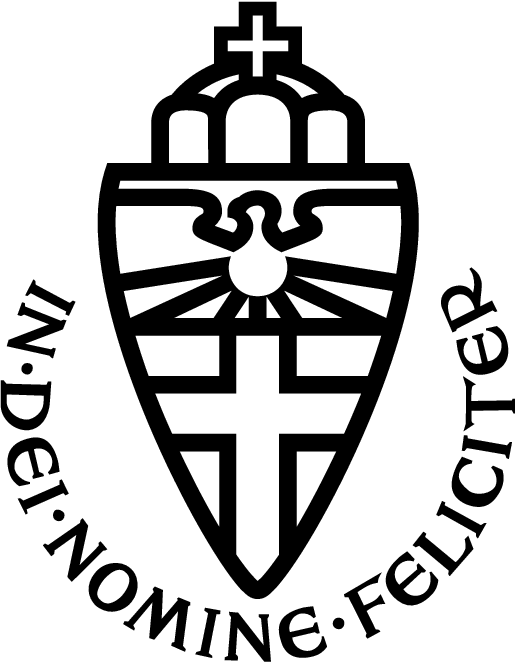}\\[.5cm]
	\textsc{Faculty of Science}\\[0.5cm]
	
	\HRule \\[0.4cm]
	{ \huge \bfseries \thesistitle}\\[0.1cm]
	\textsc{\thesissubtitle}\\
	\HRule \\[.5cm]
	\textsc{\large Thesis MSc Mathematics}\\[.5cm]
	
	\begin{minipage}{0.4\textwidth}
	\begin{flushleft} \large
	\emph{Author:}\\
	\thesisauthorfirst\space \textsc{\thesisauthorsecond}
	\end{flushleft}
	\end{minipage}
	~
	\begin{minipage}{0.4\textwidth}
	\begin{flushright} \large
	\emph{Supervisor:} \\
	\thesissupervisorfirst\space \textsc{\thesissupervisorsecond} \\[1em]
	\emph{Second reader:} \\
	\thesissecondreaderfirst\space \textsc{\thesissecondreadersecond}
	\end{flushright}
	\end{minipage}\\[4cm]
	\vfill
	{\large \thesisdate}\\
	\clearpage
\end{titlepage}

\begin{abstract}
    The Mahler measure of a monic polynomial $P(x) = a_dx^d+a_{d-1}x^{d-1}+\dots+a_1x+a_0$ is defined as
    \begin{displaymath}
        M(P) := |a_d| \prod_{P(\alpha)=0} \max\{1, |\alpha|\},
    \end{displaymath}
    where the product runs over all roots of $P$. Lehmer's problem asks whether there exists a constant $C>1$ such that $M(P) \geq C$ for all noncyclotomic polynomials in $\mathbb{Z}[x]$. In this thesis, we examine the properties of various integer sequences related to this problem, with special focus on how these sequences might help solving Lehmer's problem.
\end{abstract}

\tableofcontents

\newpage

\section*{Acknowledgements}
First and foremost I want to thank my supervisor Wadim Zudilin for his helpful comments, suggestions and criticisms throughout the process of writing this thesis. I also want to thank all friends and relatives who have shown interest in my project.

\newpage

\section{The Mahler measure}
Let $P(x) = a_dx^d + \dots + a_1x + a_0$ be a polynomial with roots $\alpha_1,\dots,\alpha_d \in \mathbb{C}$. Then the Mahler measure of $P$ is defined as
\begin{displaymath}
    M(P) := |a_d| \prod_{j=1}^d \max\{1, |\alpha_j|\}.
\end{displaymath}
Relatedly, we also define the logarithmic Mahler measure as follows:
\begin{displaymath}
    m(P) := \log(M(P)) = \log|a_d| + \sum_{j=1}^d \max\{0, \log|\alpha_j|\}.
\end{displaymath}
While these definitions make sense for any polynomial $P \in \mathbb{C}[x]$, we shall be concerned only with monic polynomials with integer coefficients. Clearly, for such polynomials we have that $M(P) \geq 1$. Moreover, it is possible to show that $M(P) = 1$ if and only if $P$ is cyclotomic or the monomial $x$. Another property of the Mahler measure, which follows immediately from the definition, is that for any polynomials $P,Q$ we have that $M(PQ) = M(P)M(Q)$.

In 1933, Lehmer \cite{Lehmer} asked the following quesion: given $\epsilon > 0$, does there exist a monic polynomial $P$ with integer coefficients such that $1 < M(P) < 1+\epsilon$? Lehmer also found the polynomial
\begin{displaymath}
    x^{10}+x^9-x^7-x^6-x^5-x^4-x^3+x+1,
\end{displaymath}
with Mahler measure $1.17628\dots$\,. To date, this is the lowest known Mahler measure.

While Lehmer's question itself remains open, a few partial solutions and related results can be mentioned. The first one of these concerns so-called reciprocal polynomials. A polynomial $P$ of degree $d$ is said to be reciprocal (or palindromic) if $P(x) = x^dP(x^{-1})$, or equivalently, if $a_{j} = a_{d-j}$ for $j=0,\dots,d$. Note that if $P$ is a reciprocal polynomial of odd degree, then $P(-1) = 0$, so irreducible non-cyclotomic reciprocal polynomials always have even degree. The following holds:
\begin{thm}[Smyth \cite{Reciprocal}]
    Let $P \in \mathbb{Z}[x]$ be a monic polynomial with the property that $M(P) < \theta_0$, where $\theta_0 = 1.32471\dots$ is the unique real root of the polynomial $x^3-x-1$. Then $P$ is reciprocal.
\end{thm}
Clearly, the constant $\theta_0$ in this theorem is the best possible, as it equals $M(x^3-x-1)$, the Mahler measure of a non-reciprocal polynomial. Furthermore, this theorem implies that if we want to find a lower bound for $M(P)$, we can assume without loss of generality that $P$ is reciprocal. However, keeping potential other applications in mind, we will avoid assuming $P$ is reciprocal when it is not necessary to do so.

Another important result is the following, which gives a nontrivial lower bound for $m(P)$ (and thus $M(P)$) in terms of $d$:
\begin{thm}[Dobrowolski \cite{Dobr}]
    There exists a constant $c_0 > 0$ such that for any noncyclotomic polynomial $P \in \mathbb{Z}[x]$ of sufficiently high degree, we have that
    \begin{displaymath}
        m(P) \geq c_0\Bigg(\frac{\log \log d}{\log d}\Bigg)^3.
    \end{displaymath}
\end{thm}
In Dobrowolski's original proof, the constant was $1-\epsilon$, where $\epsilon > 0$ is arbitrary. He also mentioned that if we take $c_0 = \frac{1}{1200}$, the inequality holds for all $d$. Since then, the constant has been improved to $2-\epsilon$ by Cantor and Straus, and to $\frac 94-\epsilon$ by Louboutin. Note that Dobrowolski's lower bound for $m(P)$ tends to zero as $d \to \infty$, so it does not fully solve Lehmer's problem.

\section{Main results}
Throughout this thesis, $P(x) = x^d + a_{d-1}x^{d-1} + \dots + a_1x + a_0$ will be a monic polynomial over $\mathbb{Z}$ with roots $\alpha_1,\dots,\alpha_d \in \mathbb{C}$. For any integer $n \geq 1$, we define
\begin{displaymath}
    P_n(x) = \prod_{j=1}^d (x-\alpha_k^n),
\end{displaymath}
which is again a monic polynomial over $\mathbb{Z}$. Let $\Delta(P_n)$ denote the discriminant of $P_n$.

Our first main result (see Theorem \ref{Un}) is a generalisation of Dobrowolski's observation that for any prime $p$, the resultant $\Res(P_p,P)$ is divisible by $p$.

Next, we will prove a number of results about the sequence $\{\Delta(P_n)\}$.
\begin{thm}
    The following holds:
    \begin{displaymath}
        \sum_{m \mid n} \mu\Big(\frac nm\Big) \Delta(P_m) \equiv 0 \mod n,
    \end{displaymath}
    where $\mu$ is the Möbius function.
\end{thm}
Congruences like these are known as Gauss congruences.

\begin{thm}
    For any monic polynomial $P \in \mathbb{Z}[x]$ the following holds:
    \begin{displaymath}
        M(P) \geq \limsup_{n\to\infty} |\Delta(P_n)|^\frac{1}{2nd}.
    \end{displaymath}
\end{thm}

\begin{thm}
    Let $P$ be a monic polynomial with distinct roots, and suppose that $m$ of the roots lie outside of the unit circle. Then
    \begin{displaymath}
        \limsup_{n\to\infty} |\Delta(P_n)|^{\frac{1}{n}} \geq M(P)^d|a_0|^{d-m-1} \geq M(P)^{d-1}.
    \end{displaymath}
\end{thm}
Here, the second inequality follows directly from the definition of the Mahler measure.

These two theorems together imply that Lehmer's problem can be restated in terms of the sequence $\{\Delta(P_n)\}$.
\begin{thm}
    The following are equivalent:
    \begin{itemize}
        \item There exists a constant $C > 1$ such that $M(P) \geq C$ for all noncyclotomic polynomials $P \in \mathbb{Z}[x]$.
        \item There exists a constant $C > 1$ such that $\limsup_{n\to\infty} |\Delta(P_n)|^{\frac{1}{nd}} \geq C$ for all noncyclotomic polynomials $P \in \mathbb{Z}[x]$ without multiple roots.
    \end{itemize}
\end{thm}

We also prove the following result about the generating function of the sequence $\{\Delta(P_n)\}$:
\begin{thm}
    There exist polynomials $u_1,\dots,u_r \in \mathbb{Z}[z]$ and integers $c_1,\dots,c_r$ such that
    \begin{displaymath}
        \sum_{n=1}^\infty \Delta(P_n)z^n = \sum_{j=1}^r c_j\frac{zu_j'(z)}{u_j(z)}.
    \end{displaymath}
\end{thm}

\begin{thm}
    Let $P \in \mathbb{Z}[x]$ be a monic polynomial of degree $d$ with constant coefficient $a_0 = \pm 1$, and let $m \in \mathbb{Z}$. Then there exist infinitely many $n$ such that $\Delta(P_n)$ is divisible by $m^{d(d-1)}$.
\end{thm}

A similar statement holds for the sequence $\{U(n)\}$ which will be defined in section 4.

\begin{thm}
    Let $P \in \mathbb{Z}[x]$ be a monic polynomial with the property that $\Delta(P_n) \neq 0$ for all $n$. Then the quantity
    \begin{displaymath}
        \delta_n(P) := \sum_{m \mid n} \Delta(P_m)
    \end{displaymath}
    has the same sign for all $n$.
\end{thm}

\section{Proof of Dobrowolski's theorem}
In this section, we give a proof of Dobrowolski's theorem based on the proof by Cantor and Straus \cite{CantorStraus}. Before we can prove the theorem, we need a couple of lemmas.

\begin{lem} [Confluent Vandermonde determinant] \label{Vandermonde}
    For $x \in \mathbb{C}$, let $A_{n,m}(x)$ be the $n \times m$ matrix
    \begin{displaymath}
        \begin{pmatrix}
            1 & 0 & 0 & \dots & 0 \\
            x & 1 & 0 & \dots & 0 \\
            x^2 & 2x & 1 & \dots & 0 \\
            \dots & \dots & \dots & \dots & \dots \\
            x^{n-1} & (n-1)x^{n-2} & \binom{n-1}{2} x^{n-3} & \dots & \binom{n-1}{m-1} x^{n-m}
        \end{pmatrix},
    \end{displaymath}
    whose entries are given by $\frac{1}{(k-1)!} \frac{d^k}{dx^k} x^{j-1}$. Given some complex numbers $x_1,\dots,x_r$ and integers $m_1,\dots,m_r$ such that $m_1+\dots+m_r=n$, we have
    \begin{displaymath}
        \det(A_{n,m_1}(x_1) A_{n,m_2}(x_2) \dots A_{n,m_r}(x_r)) = \prod_{1 \leq j < k \leq r} (x_k-x_j)^{m_jm_k},
    \end{displaymath}
    where the $n\times n$ matrix is constructed by concatenating the matrices $A_{n,m_j}(x_j)$ for $j = 1,\dots,r$.
\end{lem}

\begin{lem}[Dobrowolski] \label{Dobrdiv}
    Let $P \in \mathbb{Z}[x]$ be a monic irreducible noncyclotomic polynomial of degree $d$ with roots $\alpha_1,\dots,\alpha_d$, and let $p$ be a prime number. Then the quantity
    \begin{displaymath}
        \prod_{j,k=1}^d (\alpha_j^p-\alpha_k) = \Res(P_p,P)
    \end{displaymath}
    is nonzero and divisible by $p^d$.
\end{lem}

\begin{lem} \label{replacepol}
    Suppose $P \in \mathbb{Z}[x]$ is a monic irreducible polynomial of degree $d$ with the property that there exists an integer $n$ and roots $\alpha, \beta$ of $P$ such that $\alpha^n = \beta^n$. Then there exists a polynomial $\tilde{P}$ of lower degree $\tilde{d}$ with $M(\tilde{P}) = M(P)$.
\end{lem}

\begin{proof}[Proof of Dobrowolski's theorem]
    Without loss of generality, we assume that $P$ is irreducible, and by Lemma \ref{replacepol}, we can also assume that $\alpha_j^n -\alpha_k^n$ never vanishes. Let $m = \log d$, $k = \lfloor \frac{m}{\log m} \rfloor$ and $s = \lfloor \frac{1}{2} (\frac{m}{\log m})^2\rfloor$, and let $\Delta$ be the confluent Vandermonde determinant constructed from the data
    \begin{displaymath}
        (x_1,\dots,x_r) = (\alpha_1,\dots,\alpha_d,\alpha_1^{p_1},\dots,\alpha_1^{p_1},\dots,\alpha_1^{p_s},\dots,\alpha_d^{p_s}),
    \end{displaymath}
    where $p_j$ denotes the $j$-th prime number, and
    \begin{displaymath}
        (m_1,\dots,m_r) = (\underbrace{k, \dots, k}_d, \underbrace{1, \dots, 1}_{sd}),
    \end{displaymath}
    so that $n = d(k+s)$. Then by Lemma \ref{Vandermonde}, $\Delta$ is divisible by the product
    \begin{displaymath}
        \prod_{j=1}^s \Res(P_{p_j})^k
    \end{displaymath}
    and hence by
    \begin{displaymath}
        \prod_{j=1}^s p_j^{kd}.
    \end{displaymath}
    Moreover, our assumptions on $P$ imply that $\Delta \neq 0$. Hence,
    \begin{displaymath}
        \Delta \geq \prod_{j=1}^s p_j^{kd}.
    \end{displaymath}
    By Hadamard's inequality, we also have
    \begin{displaymath}
        \Delta^2 \leq n^{d(k^2+s)} M(P)^{2(k+\sum_{j=1}^s p_j)n},
    \end{displaymath}
    so
    \begin{displaymath}
        n^{d(k^2+s)} M(P)^{2(k+\sum_{j=1}^s p_j)n} \geq \prod_{j=1}^s p_j^{2dk}.
    \end{displaymath}
    Taking logarithms, we get
    \begin{align*}
        m(P) &\geq \frac{2dk\sum_{j=1}^s \log p_j - d(k^2+s)\log n}{2(k+\sum_{j=1}^s p_j)n} \\
        &= \frac{2k\sum_{j=1}^s \log p_j - (k^2+s)\log n}{2(k+\sum_{j=1}^s p_j)(k+s)}.
    \end{align*}
    By the asymptotic formulae for prime numbers
    \begin{displaymath}
        \sum_{j=1}^s p_j = \frac{1}{2}s^2\log s(1+o(1)) = \frac{m^4}{4\log^3 m}(1+o(1))
    \end{displaymath}
    and
    \begin{displaymath}
        \sum_{j=1}^s \log p_j = s\log s(1+o(1)) = \frac{m^2}{\log m}(1+o(1)),
    \end{displaymath}
    we now have
    \begin{align*}
        m(P) &\geq \frac{\frac{2m^3}{\log^2m}(1+o(1)) - \frac{3}{2}\frac{m^3}{\log^2 m}(1+o(1))}{\frac{m^4}{2\log^3m}\frac{m^2}{2\log^2 m}(1+o(1))} \\
        &= \frac{2\log^3m}{m^3}(1+o(1)) = 2\Bigg(\frac{\log \log d}{\log d}\Bigg)^3(1+o(1)). \qedhere
    \end{align*}
\end{proof}

The lower bound for $\Delta$ we used in this proof is not very sharp, because we ignored many factors. Since
\begin{displaymath}
    \Delta^2 = \Delta(P)^k \prod_{j=1}^s \Res(P_{p_j},P)^{2k} \prod_{j=1}^s \Delta(P_{p_j}) \prod_{1 \leq j < l \leq s} \Res(P_{p_j},P_{p_l})^2,
\end{displaymath}
one might hope to obtain better bounds for $m(P)$ by finding nontrivial bounds for the discriminants and resultants occurring in this product.

\section{A generalisation of Dobrowolski's divisibility}
In what follows, $P(x) = \prod_{j=1}^d (x-\alpha_j) = \sum_{k=0}^d a_kx^k \in \mathbb{Z}[x]$ will be a monic irreducible polynomial of degree $d$. For each $n \in \mathbb{Z}_{>0}$, define $P_n(x) = \prod_{j=1}^d (x-\alpha_j^n) \in \mathbb{Z}[x]$. Dobrowolski's Lemma \ref{Dobrdiv} says that when $p$ is prime, the quantity
\begin{displaymath}
    U(p) = U(P;p) := \prod_{j_1=1}^d \prod_{j_2=1}^d (\alpha_{j_1}^p-\alpha_{j_2}) = \operatorname{Res}(P_p,P)
\end{displaymath}
is divisible by $p^d$. This property is both a consequence and a generalisation of Fermat's congruence $a^p - a \equiv 0 \mod p$, which holds for any $a \in \mathbb{Z}$. It is easy to see that Dobrowolski's divisibility fails when we replace $p$ by a composite number $n$. However, it is possible to find a generalisation which does hold for composite numbers as well. The starting point is Gauss's congruence
\begin{displaymath}
    \sum_{m\mid n} a^m\mu\Big(\frac nm\Big) \equiv 0 \mod n,
\end{displaymath}
valid for any $n \in \mathbb{Z}_{>0}$ and $a \in \mathbb{Z}$, which specialises to Fermat's congruence when $n = p$. If we set
\begin{displaymath}
    \mathcal{D}_+(n) := \Big\{m\mid n: \mu\Big(\frac nm\Big) = 1\Big\} \text{ and } \mathcal{D}_-(n) := \Big\{m \mid n: \mu\Big(\frac nm\Big) = -1\Big\},
\end{displaymath}
then Gauss's congruence takes the form
\begin{displaymath}
    \sum_{m \in \mathcal{D}_+} a^m - \sum_{m \in \mathcal{D}_-} a^m \equiv 0 \mod n,
\end{displaymath}
which suggests the following generalisation of Dobrowolski's divisibility. For a multiset $\mathcal{N} = \{n_1,\dots,n_r\} \subseteq \mathbb{Z}_{>0}$, where entries may occur multiple times, define the polynomial
\begin{displaymath}
    P_\mathcal{N}(x) := \prod_{j_1,\dots,j_r}^d(x-\alpha_{j_1}^{n_1}\dots\alpha_{j_r}^{n_r}) \in \mathbb{Z}[x],
\end{displaymath}
which has degree $d^r$. If we assume $n > 1$ and let $n = p_1^{k_1}\dots p_l^{k_l}$ be its prime factorisation, then the sets $\mathcal{D}_+(n)$ and $\mathcal{D}_-(n)$ both have $r = 2^{l-1}$ elements. Define
\begin{displaymath}
    U(n) = U(P; n) := \operatorname{Res}(P_{\mathcal{D}_+(n)},P_{\mathcal{D}_-(n)}).
\end{displaymath}
Then for prime numbers $p$, this definition coincides with $U(p)$ as defined earlier. Moreover, the following holds:
\begin{thm} \label{Un}
    Let $P \in \mathbb{Z}[x]$ be a monic polynomial of degree $d$ with roots $\alpha_1,\dots,\alpha_d$, and let $n >1$ be an integer with $l$ distinct prime factors. Then $U(n)$ is divisible by $n^{d^r}$, where $r=2^{l-1}$.
\end{thm}

In order to prove this theorem, we need a few lemmas.

\begin{lem} \label{Gausspol}
    For any $n \in \mathbb{Z}_{>1}$, the following congruence holds:
    \begin{displaymath}
        \sum_{m \mid n} \mu\Big(\frac nm\Big) P_m \equiv 0 \mod n.
    \end{displaymath}
\end{lem}
\begin{proof}
    We check this for each coefficient separately. In other words, we need to show that
    \begin{equation} \label{polcong}
        \sum_{m \mid n} \mu\Big(\frac nm\Big) e_k(\alpha_1^m,\dots,\alpha_d^m) \equiv 0 \mod n
    \end{equation}
    holds for $k = 0,\dots, d$, where $e_k(x_1,\dots,x_d)$ denotes the elementary symmetric polynomial of degree $k$ in $d$ variables, i.e.
    \begin{displaymath}
        e_k(x_1,\dots,x_d) = \sum_{1\leq j_1<\dots<j_k\leq d} x_{j_1}\dots x_{j_k}.
    \end{displaymath} For $k=0$, congruence \eqref{polcong} reduces to $\sum_{m\mid n} \mu(\frac nm) = 0$, while for $k=1$, the congruence was shown by Smyth in \cite{Smyth}. The case for general $k$ follows from the $k=1$ case by replacing $\{\alpha_1,\dots,\alpha_d\}$ with $\{\alpha_{j_1}\dots\alpha_{j_k}:1\leq j_1<\dots<j_k\leq d\}$.
\end{proof}

\begin{cor} \label{Unspecial}
    Theorem \ref{Un} holds when $n = p^k$ is a prime power. In other words, the quantity
    \begin{displaymath}
        U(p^k) = \prod_{j_1,j_2=1}^d \big(\alpha_{j_1}^{p^k}-\alpha_{j_2}^{p^{k-1}}\big) = \operatorname{Res}(P_{p^k},P_{p^{k-1}})
    \end{displaymath}
    is divisible by $p^{kd}$.
\end{cor}
\begin{proof}
    By Lemma \ref{Gausspol}, we have that $P_{p^k} \equiv P_{p^{k-1}} \bmod p^k$. This implies that the last $d$ columns of the Sylvester matrix of the polynomials $P_{p^k}$ and $P_{p^{k-1}}-P_{p^k}$ are divisible by $p^k$. Hence, $\Res(P_{p^k},P_{p^{k-1}}) = \Res(P_{p^k},P_{p^{k-1}}-P_{p^k})$ is divisible by $p^{kd}$.
\end{proof}

Let $\mathbb{Z}_{(p)}$ be the localisation of $\mathbb{Z}$ at the prime ideal $(p)$. Then $\mathbb{Z}_{(p)} = \mathbb{Z}_p \cap \mathbb{Q}$, where $\mathbb{Z}_p$ denotes the ring of $p$-adic integers. For polynomials $P \in \mathbb{Z}_{(p)}[x]$, we can prove statements similar to those of Lemma \ref{Gausspol} and Corollary \ref{Unspecial}. More precisely, we have:
\begin{lem} \label{Unspecialzp}
    Let $p$ be a prime and $k \in \mathbb{N}$. Then for any monic polynomial $P \in \mathbb{Z}_{(p)}[x]$, we have
    \begin{displaymath}
        P_{p^k}-P_{p^{k-1}} \in p^k\mathbb{Z}_{(p)}
    \end{displaymath}
    and
    \begin{displaymath}
        \Res(P_{p^k},P_{p^{k-1}}) \in p^{kd}\mathbb{Z}_{(p)}.
    \end{displaymath}
\end{lem}
\begin{proof}
    Consider the polynomial $\bar{P}$ given by
    \begin{displaymath}
        \bar{P}(x) = u^dP\Big(\frac xu\Big),
    \end{displaymath}
    where $u$ is an integer chosen in such a way that $\bar{P} \in \mathbb{Z}[x]$. Note that since $P \in \mathbb{Z}_{(p)}[x]$, we can choose $u$ such that $p \nmid u$. From Lemma \ref{Gausspol} it now follows that
    \begin{equation} \label{weirdpolcongruence}
        \bar{P}_{p^k} - \bar{P}_{p^{k-1}} \in p^k\mathbb{Z}[x].
    \end{equation}
    Note that the roots of $\bar{P}_n$ are of the form $u^n\alpha_j^n$, which implies that
    \begin{displaymath}
        \bar{P}_n(x) = u^{dn}P_n\Big(\frac{x}{u^n}\Big).
    \end{displaymath}
    Hence, \eqref{weirdpolcongruence} can be restated as
    \begin{displaymath}
        u^{dp^k}P_{p^k}\Big(\frac{x}{u^{p^k}}\Big) - u^{dp^{k-1}}P_{p^{k-1}}\Big(\frac{x}{u^{p^{k-1}}}\Big) \in p^k \mathbb{Z}[x].
    \end{displaymath}
    If we denote by $a_{j,p^k}$ the coefficient of $x^j$ in $P_{p^k}$ (and similarly for $P_{p^{k-1}}$), we now have
    \begin{displaymath}
        a_{j,p^k}u^{(d-j)p^k} - a_{j,p^{k-1}}u^{(d-j)p^{k-1}} \in p^k \mathbb{Z}_{(p)}[x].
    \end{displaymath}
    Since
    \begin{displaymath}
        u^{(d-j)p^k} \equiv u^{(d-j)p^{k-1}} \bmod p^k,
    \end{displaymath}
    this implies that
    \begin{displaymath}
        a_{j,p^k} - a_{j,p^{k-1}} \in p^k \mathbb{Z}_{(p)}[x],
    \end{displaymath}
    which proves the first claim. The second claim follows from the first one via the same argument as the one used in the proof of Corollary \ref{Unspecial}.
\end{proof}

To prove Theorem \ref{Un} when $n$ is not a prime power, we prove for each prime $p \mid n$ that $U(n)$ is divisible by the correct power of $p$. More precisely, if $n = p_1^{k_1}\dots p_l^{k_l}$, then we prove that $U(n)$ is divisible by $p_t^{k_td^r}$ for $t=1,\dots,l$. The proof makes use of the polynomial
\begin{displaymath}
    Q(x) = \prod_{j_1,\dots,j_{2^{l-1}}=1}^d \Bigg(x- \frac{\prod_{m_s \in \mathcal{D}_+(n')} \alpha_{j_s}^{m_s}}{\prod_{m_s \in \mathcal{D}_-(n')} \alpha_{j_s}^{m_s}} \Bigg),
\end{displaymath}
where $n' = n/p_t^{k_t}$ and the divisors of $n'$ are ordered such that $\mathcal{D}_+(n') = \{m_1,\dots,m_{2^{l-2}}\}$ and $\mathcal{D}_-(n') = \{m_{2^{l-2}+1},\dots,m_{2^{l-1}}\}$. Write $Q(x) = \sum_{k=0}^{d^r} b_kx^k$.

\begin{lem} \label{Qint}
    If $p \nmid a_0$, then $Q \in \mathbb{Z}_{(p)}[x]$.
\end{lem}
\begin{proof}
    Note that the coefficients of $Q$ are rational because $Q$ is invariant under the action of the Galois group $\operatorname{Gal}(\mathbb{Q}(\alpha_1,\dots,\alpha_d)/\mathbb{Q})$. Moreover, since $P$ is monic and $p \nmid a_0$, all roots of $P$ satisfy $|\alpha_j|_p = 1$. Hence, if $\beta$ is any root of $Q$, then $|\beta|_p = 1$ as well. Since the coefficients of $Q$ are symmetric polynomials in the roots of $Q$, this implies that $|b_k|_p \leq 1$ for all $k$, which means precisely that the denominator of $b_k$ is not divisible by $p$ for any $k$.
\end{proof}

\begin{cor} \label{Ungcd}
    Let $a_0$ be the constant coefficient of $P$, and assume $(n, a_0) = 1$. Then $U(n)$ is divisible by $n^{d^r}$.
\end{cor}
\begin{proof}
    As indicated, we will prove the claim by showing that $U(n)$ is divisible by $p_t^{k_td^r}$ for $t=1,\dots,l$. Fix $t \in \{1,\dots,l\}$ and let $n' := \frac{n}{p^k}$, where we have omitted the index $t$ for simplicity. Then we have that
    \begin{align*}
        \mathcal{D}_+(n) &= \{p^k m: m \in \mathcal{D}_+(n')\} \cup \{p^{k-1}m:m \in \mathcal{D}_-(n')\}, \\
        \mathcal{D}_-(n) &= \{p^k m: m \in \mathcal{D}_-(n')\} \cup \{p^{k-1}m:m \in \mathcal{D}_+(n')\}.
    \end{align*}
    Hence, $U(n)$ is the product of all elements of the form
    \begin{align*}
        &\prod_{m_s \in \mathcal{D}_+(n')} \alpha_{j_s}^{p^{k} m_s}\prod_{m_s \in \mathcal{D}_-(n')} \alpha_{j_s}^{p^{k-1}m_s} - \prod_{m_s \in \mathcal{D}_-(n')} \alpha_{j'_s}^{p^k m_s}\prod_{m_s \in \mathcal{D}_+(n')} \alpha_{j'_s}^{p^{k-1}m_s} \\
        = &\prod_{m_s \in \mathcal{D}_-(n')} \big(\alpha_{j_s}^{p^{k-1}m_s}\alpha_{j'_s}^{p^km_s}\big) \Bigg( \frac{\prod_{m_s \in \mathcal{D}_+(n')} \alpha_{j_s}^{p^km_s}}{\prod_{m_s\in\mathcal{D}_-(n')} \alpha_{j'_s}^{p^km_s}} - \frac{\prod_{m_s \in \mathcal{D}_+(n')} \alpha_{j'_s}^{p^{k-1}m_s}}{\prod_{m_s \in \mathcal{D}_-(n')} \alpha_{j_s}^{p^{k-1}m_s}} \Bigg).
    \end{align*}
    Since
    \begin{displaymath}
        \prod_{j_s,j'_s=1}^d \prod_{m_s \in \mathcal{D}_-(n')} \big(\alpha_{j_s}^{p^{k-1}m_s}\alpha_{j'_s}^{p^km_s}\big) = a_0^{(p^k+p^{k-1})\sum_{m \in \mathcal{D}_-(n')}m} =: N,
    \end{displaymath}
    and
    \begin{displaymath}
        \prod_{j_s,j'_s=1}^d \Bigg( \frac{\prod_{m_s \in \mathcal{D}_+(n')} \alpha_{j_s}^{p^km_s}}{\prod_{m_s\in\mathcal{D}_-(n')} \alpha_{j'_s}^{p^km_s}} - \frac{\prod_{m_s \in \mathcal{D}_+(n')} \alpha_{j'_s}^{p^{k-1}m_s}}{\prod_{m_s \in \mathcal{D}_-(n')} \alpha_{j_s}^{p^{k-1}m_s}} \Bigg) = \Res(Q_{p^k}, Q_{p^{k-1}}),
    \end{displaymath}
    (in both cases the product ranges over all possible values of all indices $j_s$ and $j'_s$), we have that
    \begin{displaymath}
        U(n) = N\cdot \Res(Q_{p^k}, Q_{p^{k-1}}).
    \end{displaymath}
    Since $(n,a_0) = 1$, Lemma \ref{Qint} implies that $Q \in \mathbb{Z}_{(p)}[x]$, so we can use Corollary \ref{Unspecialzp} to conclude that $p^{k d^r} \mid U(n)$. Since the same argument can be applied to all prime factors of $n$, it follows that
    \begin{displaymath}
        n^{d^r} \mid U(n). \qedhere
    \end{displaymath}
\end{proof}

\begin{rmk}
    From the proof it also follows that if $m \mid n$ and $(m,a_0) = 1$, then $U(n)$ is divisible by $m^{d^r}$.
\end{rmk}

To finish the proof of Theorem \ref{Un}, we have to deal with the case $(n, a_0) > 1$. We do this by showing that $U(n)$ is always divisible by some power of $a_0$.

\begin{thm} \label{Unconstant}
    Let $n = p_1^{k_1}\dots p_l^{k_l}$ be the prime factorisation of $n$, and define $n'= p_1^{k_1-1}\dots p_l^{k_l-1}$. Then $U(n)$ is divisible by
    \begin{displaymath}
        a_0^{n'2^{2(l-1)}d^{2^l-2}}
    \end{displaymath}
    where $a_0$ is the constant coefficient of $P$ and $l$ is the number of distinct prime factors of $n$.
\end{thm}
\begin{proof}
    We give a proof for the special case where $n=p^2q^2$ is the product of two squares of primes, together with some comments about how this idea can be applied to the general case. We have $n' = pq$ and
    \begin{displaymath}
        U(n) = \prod_{j_1,j_2,j_3,j_4=1}^d \big(\alpha_{j_1}^n\alpha_{j_2}^{n'} - \alpha_{j_3}^{p^2q}\alpha_{j_4}^{pq^2}\big).
    \end{displaymath}
    Note that every exponent appearing in this product is divisible by $n'$, and hence greater than or equal to it. Now consider the factors which have an equal index appearing on both sides of the minus sign, e.g.
    \begin{displaymath}
        U(n) = \prod_{j_1,j_2,j_3=1}^d \big(\alpha_{j_1}^n\alpha_{j_2}^{n'} - \alpha_{j_3}^{p^2q}\alpha_{j_2}^{pq^2}\big),
    \end{displaymath}
    where we have taken $j_2=j_4$. Since this product has $d^3$ factors, it can be rewritten as
    \begin{displaymath}
        \pm a_0^{n'd^2} \prod_{j_1,j_2,j_3=1}^d \big(\alpha_{j_1}^n - \alpha_{j_3}^{p^2q}\alpha_{j_2}^{pq^2-n'}\big),
    \end{displaymath}
    which shows that this factor is divisible by $a_0^{n'd^2}$. In the general case, this product contains $2^l-1$ distinct indices, so that it is divisible by $a_0^{n'd^{2^l-2}}$.

    Above, we took $j_2=j_4$, but we could also have chosen another pair of indices. Since there are two indices on both sides of the minus sign, there are four ways to choose such a pair of indices. Hence, in this special case, $U(n)$ is divisible by $a_0^{4n'd^2}$.
    
    In the general case, there are $2^{l-1}$ indices on each side of the minus sign, so there are $2^{2(l-1)}$ ways of making an equal index appear on both sides. Hence, we can construct $2^{2(l-1)}$ factors of $U(n)$ which are divisible by $a_0^{n'd^{2^l-2}}$, so $U(n)$ is divisible by
    \begin{displaymath}
        a_0^{n'2^{2(l-1)}d^{2^l-2}}. \qedhere
    \end{displaymath}
\end{proof}

\begin{rmk}
    While this result is already much stronger than necessary for our purposes, it is possible to show that $U(n)$ is divisible by an even higher power of $a_0$.
\end{rmk}

\begin{proof}[Proof of Theorem \ref{Un}]
    The case $l=1$ was covered in Corollary \ref{Unspecial}, so we assume $l \geq 2$. The proof of Corollary \ref{Ungcd} implies that $U(n)$ is divisible by $p^{kd^r}$ for all prime powers $p^k \mid n$ for which $p \nmid a_0$. Now suppose $p^k \mid n$ and $p \mid a_0$. Then by Theorem \ref{Unconstant}, $U(n)$ is divisible by
    \begin{displaymath}
        a_0^{n'2^{2(l-1)}d^{2^l-2}}.
    \end{displaymath}
    Since $p \mid a_0$ and $p^{k-1} \mid n'$, this is divisible by
    \begin{displaymath}
        p^{p^{k-1}2^{2(l-1)}d^{2^l-2}}.
    \end{displaymath}
    Since $p^{k-1} \geq k$, $2^{2(l-1)} \geq 1$ and $2^l-2 \geq 2^{l-1}$ (here we used the assumption that $l \geq 2$), this is divisible by
    \begin{displaymath}
        p^{kd^{2^{l-1}}}. \qedhere
    \end{displaymath}
\end{proof}

It is possible to bound $U(n)$ from above using the Mahler measure.
\begin{thm} \label{Unupperbound}
    Let $P \in \mathbb{Z}[x]$ be a monic polynomial of degree $d$ and let $n \in \mathbb{Z}_{>1}$. Then
    \begin{displaymath}
        |U(n)| \leq 2^{d^{2^l}} M(P)^{d^{2^l}\sum_{m\in\mathcal{D}_+(n)}m},
    \end{displaymath}
    where $l$ is the number of distinct prime factors of $n$.
\end{thm}
\begin{proof}
    Each factor of $U(n)$ is of the form
    \begin{displaymath}
        \prod_{m \in \mathcal{D}_+(n)}\alpha_{j_m}^m-\prod_{m\in\mathcal{D}_-(n)}\alpha_{j_m}^m,
    \end{displaymath}
    and the absolute value of this can be estimated as follows:
    \begin{align*}
        \Bigg|\prod_{m \in \mathcal{D}_+(n)}\alpha_{j_m}^m-\prod_{m\in\mathcal{D}_-(n)}\alpha_{j_m}^m\Bigg| &\leq 2\max\Bigg\{\prod_{m \in \mathcal{D}_+(n)}\alpha_{j_m}^m, \prod_{m\in\mathcal{D}_-(n)}\alpha_{j_m}^m\Bigg\} \\
        &\leq 2M(P)^{\sum_{m\in\mathcal{D}_+(n)}m},
    \end{align*}
    where we have used that each root has absolute value $\leq M(P)$. Since $U(n)$ has $d^{2^l}$ factors in total, the result follows.
\end{proof}

One might try to use this estimate together with the divisibility of $U(n)$ to derive lower bounds for the Mahler measure. However, one runs into two problems when doing this: firstly, $U(n)$ may be zero, and secondly, even if $U(n) \neq 0$, the bound obtained from combining Theorem \ref{Unupperbound} and Theorem \ref{Un} is weaker than Dobrowolski's.

\section{Gauss congruences for discriminants and resultants}
From Lemma \ref{Gausspol} we can deduce the following:
\begin{thm}
    Let $P \in \mathbb{Z}[x]$ be a monic polynomial (not necessarily irreducible). Then the following holds:
    \begin{displaymath}
        \sum_{m \mid n} \mu\Big(\frac nm\Big) \Delta(P_m) \equiv 0 \mod n.
    \end{displaymath}
\end{thm}
\begin{proof}
    If $n = p^k$, then Lemma \ref{Gausspol} says that
    \begin{displaymath}
        P_{p^k} \equiv P_{p^{k-1}} \mod p^k.
    \end{displaymath}
    From the explicit expression of the discriminant of a polynomial in terms of its coefficients it now follows that
    \begin{displaymath}
        \Delta(P_{p^k}) \equiv \Delta(P_{p^{k-1}}) \mod p^k.
    \end{displaymath}
    To prove the claim for general $n$, we once again use the approach of proving it for each prime divisor of $n$ separately. Let $p \mid n$ and write $n = p^kn'$ with $p \nmid n'$. Then we can form the sets
    \begin{align*}
        D_1 &:= \{m \mid n: \mu\Big(\frac nm\Big) = 1, p^k \mid m\}, \\
        D_2 &:= \{m \mid n: \mu\Big(\frac nm\Big) = -1, p^k \nmid m\}, \\
        D_3 &:= \{m \mid n: \mu\Big(\frac nm\Big) = -1, p^k \mid m\}, \\
        D_4 &:= \{m \mid n: \mu\Big(\frac nm\Big) = 1, p^k \nmid m\}.
    \end{align*}
    Note that the requirement that $\mu(n/m) \neq 0$ forces $m$ to be divisible by $p^{k-1}$. Lemma \ref{Gausspol} now reads
    \begin{displaymath}
        \sum_{m \in D_1} P_m - \sum_{m \in D_2} P_m - \sum_{m \in D_3} P_m + \sum_{m \in D_4} P_m \equiv 0 \mod n.
    \end{displaymath}
    Since every element of $D_1$ is of the form $pm$ for some $m \in D_2$ and every element of $D_3$ is of the form $pm$ for some $m \in D_4$, the left-hand side can be rewritten as
    \begin{displaymath}
        \sum_{m \in D_2} (P_{pm} - P_m) - \sum_{m \in D_4} (P_{pm} - P_{m}).
    \end{displaymath}
    Because every $m$ in the above expression is divisible by $p^{k-1}$, we can rewrite it as
    \begin{displaymath}
        \sum_{m \in D_2'} (P_{p^km} - P_{p^{k-1}m}) - \sum_{m \in D_4'} (P_{p^km}-P_{p^{k-1}m}),
    \end{displaymath}
    where $D_i' := \{m \in \mathbb{Z}: p^{k-1}m \in D_i\}$. We have already seen that
    \begin{displaymath}
        \Delta(P_{p^km}) - \Delta(P_{p^{k-1}m}) \equiv 0 \mod p^k.
    \end{displaymath}
    This implies that
    \begin{align*}
        \sum_{m \mid n} \mu\Big(\frac nm\Big) \Delta(P_m) &= \sum_{m \in D_2'} (\Delta(P_{p^km}) - \Delta(P_{p^{k-1}m})) - \sum_{m \in D_4'} (\Delta(P_{p^km})-\Delta(P_{p^{k-1}m})) \\
        &\equiv 0 \mod p^k.
    \end{align*}
    Since this argument can be applied to every prime divisor of $n$, the theorem follows.
\end{proof}

While these congruences are interesting in their own right, they do not seem to be immediately applicable to the problem of bounding the Mahler measure. This is because the congruences do not imply any lower bounds for the sequence $\{\Delta(P_n)\}$.

Using a similar argument, one can also prove the following:
\begin{thm}
    Let $P,Q \in \mathbb{Z}[x]$ be monic polynomials. Then
    \begin{displaymath}
        \sum_{m \mid n} \mu\Big(\frac{n}{m}\Big) \Res(P_m,Q) \equiv 0 \bmod n.
    \end{displaymath}
\end{thm}

\section{Estimates for $\Delta(P_n)$ and their application to the Mahler measure}
The following lemma suggests that finding good estimates for $|\Delta(P_n)|$ might be useful for proving better bounds for the Mahler measure. In particular, if the constant $C$ can be chosen independently of $P$, then the lemma implies a solution to Lehmer's problem.

\begin{lem} \label{mpgeqc}
    Let $P$ be a monic polynomial, and suppose there exists a constant $C>1$, such that $|\Delta(P_n)| > C^{nd}$ for infinitely many $n$. Then we have
    \begin{displaymath}
        m(P) \geq \frac{\log C}{2}.
    \end{displaymath}
\end{lem}
\begin{proof}
    For any $n \in \mathbb{Z}_{>0}$, consider the determinant
    \begin{displaymath}
        \det
        \begin{pmatrix}
            1 & 1 & \dots & 1 \\
            \alpha_1^n & \alpha_2^n & \dots & \alpha_d^n \\
            \vdots & \vdots & & \vdots \\
            \alpha_1^{n(d-1)} & \alpha_2^{n(d-1)} & \dots & \alpha_d^{n(d-1)}
        \end{pmatrix}.
    \end{displaymath}
    The square of this determinant equals $\Delta(P_n)$. By Hadamard's inequality, we have
    \begin{displaymath}
        |\Delta(P_n)| \leq \prod_{j=1}^d \sum_{k=1}^d |\alpha_j|^{2n(k-1)} \leq d^d M(P)^{2nd}.
    \end{displaymath}
    If we choose $n$ such that $|\Delta(P_n)| > C^{nd}$, we see that
    \begin{displaymath}
        m(P) \geq \frac{nd \log C - d\log d}{2nd} = \frac{n\log C - \log d}{2n}.
    \end{displaymath}
    Since we can choose $n$ to be arbitrarily large, the result follows. 
\end{proof}

\begin{cor} \label{upperboundlimsup}
    For any monic polynomial $P \in \mathbb{Z}[x]$ the following holds:
    \begin{displaymath}
        M(P) \geq \limsup_{n\to\infty} |\Delta(P_n)|^\frac{1}{2nd}.
    \end{displaymath}
\end{cor}
\begin{proof}
    For any $C < \limsup_{n\to\infty} |\Delta(P_n)|^{\frac{1}{nd}}$, we have
    \begin{displaymath}
        M(P) \geq C^{\frac{1}{2}}.
    \end{displaymath}
    Taking the limit as $C \to \limsup_{n\to\infty} |\Delta(P_n)|^{\frac{1}{nd}}$, the result follows.
\end{proof}

\subsection{Estimates for low degrees}
For quadratic polynomials, the growth of $\Delta(P_n)$ is easy to study directly, because in that case, $\Delta(P_n)$ is given by the following simple expression:
\begin{displaymath}
    \Delta(P_n) = (\alpha_1^n-\alpha_2^n)^2.
\end{displaymath}
Excluding cyclotomic and reducible polynomials, there are now three possibilities:
\begin{itemize}
    \item $|\alpha_1| \neq |\alpha_2|$, say $|\alpha_1| > |\alpha_2|$.
    \item $|\alpha_1| = |\alpha_2|$ and $\alpha_1 = -\alpha_2$.
    \item $|\alpha_1| = |\alpha_2|$ and $\alpha_1 \neq -\alpha_2$ (i.e. $\alpha_1 = \overline{\alpha_2}$).
\end{itemize}
Note that the assumption that $P$ be non-cyclotomic implies that $|\alpha_1| > 1$ in all three cases. In the first case, we have $\Delta(P_n) \sim \alpha_1^{2n}$, which implies that we can take $C = |\alpha_1|-\epsilon$ in Lemma \ref{mpgeqc}, for any $\epsilon > 0$.

In the second case, we have $\Delta(P_n) = 0$ for even $n$ and $\Delta(P_n) = 4\alpha_1^{2n}$ for odd $n$. Thus, we can take $C = |\alpha_1|$.

In the third case, the behaviour of $\Delta(P_n)$ is more complicated, but it is still possible possible to find a constant $C$ satisfying the condition of Lemma \ref{mpgeqc}. Write $\alpha_1$ in polar coordinates as $\alpha_1 = re^{i\theta}$ with $\theta \in (0, \pi)$. Then
\begin{displaymath}
    \Delta(P_n) = (r^ne^{in\theta}-r^ne^{-in\theta})^2 = -4r^{2n}\sin^2 n\theta.
\end{displaymath}
Since there exist infinitely many $n$ such that $n\theta \in (\frac{\pi-\theta}{2},\frac{\pi+\theta}{2}) \mod 2\pi$, this implies that there exist infinitely many $n$ such that
\begin{displaymath}
    |\Delta(P_n)| > 4r^{2n}\sin^2\Big(\frac{\pi-\theta}{2}\Big),
\end{displaymath}
which implies that we can again take $C = r-\epsilon = |\alpha_1|-\epsilon$.

\subsection{A lower bound for $|\Delta(P_n)|$ for infinitely many $n$}
The ideas of the preceding section can be used to find a lower bound for $|\Delta(P_n)|$ in terms of $M(P)$ which holds for infinitely many $n$. More precisely, the following holds:
\begin{thm} \label{lowerboundlimsup}
    Let $P(x) = x^d + \dots + a_1x + a_0 \in \mathbb{Z}[x]$ be a monic polynomial with distinct roots $\alpha_1,\dots,\alpha_d$, and suppose that $m$ of the roots lie outside of the unit circle. Then
    \begin{displaymath}
        \limsup_{n\to\infty} |\Delta(P_n)|^{\frac{1}{n}} \geq M(P)^d|a_0|^{d-m-1} \geq M(P)^{d-1}.
    \end{displaymath}
\end{thm}

To prove this theorem, we need the following two lemmas:
\begin{lem}[\cite{Fourier}, Theorem 4.2.1] \label{eqdistlem}
    Let $x \in \mathbb{R}$ be irrational. Then the sequence $\{nx \bmod \mathbb{Z}\}$ is equidistributed in $[0,1)$. In other words, for each subinterval $I \subseteq [0,1)$, we have
    \begin{displaymath}
        \lim_{n\to\infty} \frac{\# \{k \leq n: kx \in I \bmod \mathbb{Z}\}}{n} = l(I),
    \end{displaymath}
    where $l(I)$ denotes the length of the interval.
\end{lem}

\begin{lem} \label{equidistribution}
    Let $x_1,\dots,x_m$ be real numbers that are distinct modulo $\mathbb{Z}$, and let $\|x_j\|$ denote the distance from $x_j$ to the nearest integer. Then there exists a constant $\delta > 0$ such that the inequality
    \begin{equation} \label{deltaeq}
        \min \{\|nx_j\|: 1 \leq j \leq m\} > \delta
    \end{equation}
    holds for infinitely many $n$.
\end{lem}
\begin{proof}
    First assume all $x_j$ are irrational, and take any $\delta < \frac{1}{2m}$ (for example, we may choose $\delta = \frac{1}{3m}$). Consider the sets
    \begin{displaymath}
        A_j = \{n \in \mathbb{N}: \|nx_j\| \leq \delta\}.
    \end{displaymath}
    By Lemma \ref{eqdistlem}, each of these sets has asymptotic density $2\delta$. Hence, their union has upper asymptotic density $\leq 2m\delta < 1$, which implies that the complement of $\bigcup_{j=1}^m A_j$ has positive lower asymptotic density, and hence infinitely many elements. However, $\mathbb{N}\backslash\bigcup_{j=1}^m A_j$ is exactly the set of $n \in \mathbb{N}$ for which \eqref{deltaeq} holds, so the statement follows.

    If any of the $x_j$ is rational, we have to be a bit more careful as we cannot apply Lemma \ref{eqdistlem} in that case. However, we still have the bound $d(A_j) \leq 3\delta$, which implies that we can repeat the above argument with $\delta < \frac{1}{3m}$.
\end{proof}

\begin{proof}[Proof of Theorem \ref{lowerboundlimsup}]
    Order the roots $\alpha_1,\dots,\alpha_d$ in such a way that $|\alpha_j| \geq |\alpha_k|$ for $j < k$, and consider the individual factors of the product
    \begin{displaymath}
        \Delta(P_n) = \prod_{1 \leq j < k \leq d} (\alpha_j^n-\alpha_k^n)^2.
    \end{displaymath}
    If $|\alpha_j| \neq |\alpha_k|$, then clearly $|\alpha_j^n-\alpha_k^n|^2 \sim |\alpha_j|^{2n}$ as $n \to \infty$. Now suppose $|\alpha_j| = |\alpha_k| = r$. Then there exist $\theta_j,\theta_k \in \mathbb{R}$ such that $\alpha_j = re^{i\theta_j}$ and $\alpha_k = re^{i\theta_k}$, which implies that
    \begin{displaymath}
        |\alpha_j^n-\alpha_k^n|^2 = r^{2n}|e^{in\theta_j}-e^{in\theta_k}|.
    \end{displaymath}
    Hence, we need to show that there exists a constant $c > 0$ such that there are infinitely many $n$ for which the inequality $|e^{in\theta_j}-e^{in\theta_k}| > c$ holds for all pairs of roots with the same absolute value. Applying Lemma \ref{equidistribution} to the numbers of the form $\frac{\theta_j-\theta_k}{2\pi}$, of which there are at most $\binom{d}{2}$, we see that there exists a constant $\delta >0$ such that $\min_{j\neq k}\|n\theta_j-\theta_k\| > \delta$ for infinitely many $n$, which in turn implies that
    \begin{displaymath}
        \min_{j \neq k} |e^{in\theta_j}-e^{in\theta_k}| > |e^{2\pi i \delta}-1|
    \end{displaymath}
    for infinitely many $n$.

    Multiplying the lower bounds for all factors of $\Delta(P_n)$ together, we see that there exists some constant $\tilde{c} > 0$ such that
    \begin{displaymath}
        |\Delta(P_n)| > \tilde{c} \prod_{j=1}^d |\alpha_j|^{2n(d-j)}
    \end{displaymath}
    holds for infinitely many $n$. This implies that
    \begin{equation} \label{limsupineq}
        \limsup_{n\to\infty} |\Delta(P_n)|^{\frac{1}{n}} \geq \prod_{j=1}^d |\alpha_j|^{2(d-j)}.
    \end{equation}
    Since
    \begin{displaymath}
        \prod_{j=1}^d \alpha_j = a_0,
    \end{displaymath}
    we have
    \begin{displaymath}
        \prod_{j=m+1}^d \alpha_j = a_0\prod_{j=1}^m \alpha_j^{-1},
    \end{displaymath}
    which implies that
    \begin{displaymath}
        \prod_{j=m+1}^d |\alpha_j|^{2(d-j)} \geq \prod_{j=m+1}^d |\alpha_j|^{d-m-1} = |a_0|^{d-m-1}\prod_{j=1}^m |\alpha_j|^{-d+m+1}.
    \end{displaymath}
    Hence,
    \begin{align*}
        \prod_{j=1}^d |\alpha_j|^{2(d-j)} &\geq |a_0|^{d-m-1} \prod_{j=1}^m |\alpha_j|^{d-2j+m+1} \\
        &\geq |a_0|^{d-m-1} \prod_{j=1}^m |\alpha_j|^d \\
        &= |a_0|^{d-m-1} M(P)^d.
    \end{align*}
    Since $m \leq d$ and $M(P) \geq |a_0|$, the inequality
    \begin{displaymath}
        |a_0|^{d-m-1} M(P)^d \geq |a_0|^{-1}M(P)^d \geq M(P)^{d-1}
    \end{displaymath}
    follows.
\end{proof}

The inequality in \eqref{limsupineq} is actually an equality.
\begin{thm}
    Let $P \in \mathbb{Z}[x]$ be a polynomial without multiple roots. Then
    \begin{displaymath}
        \limsup_{n\to\infty} |\Delta(P_n)|^{\frac{1}{n}} = \prod_{j=1}^d |\alpha_j|^{2(d-j)}.
    \end{displaymath}
\end{thm}
\begin{proof}
    We saw in the proof of Theorem \ref{lowerboundlimsup} that
    \begin{displaymath}
        \limsup_{n\to\infty} |\Delta(P_n)|^{\frac{1}{n}} \geq \prod_{j=1}^d |\alpha_j|^{2(d-j)}.
    \end{displaymath}
    We use the same approach of estimating the individual factors to find an upper bound. Since $|\alpha_j^n-\alpha_k^n| \leq 2|\alpha_j|^n$, we have
    \begin{displaymath}
        |\Delta(P_n)| \leq 2^{d(d-1)} \prod_{j=1}^d |\alpha_j|^{2n(d-j)},
    \end{displaymath}
    which implies that
    \begin{displaymath}
        \limsup_{n\to\infty} |\Delta(P_n)|^{\frac{1}{n}} \leq \prod_{j=1}^d |\alpha_j|^{2(d-j)}. \qedhere
    \end{displaymath}
\end{proof}

Combining the inequalities from Corollary \ref{upperboundlimsup} and Theorem \ref{lowerboundlimsup}, we get
\begin{displaymath}
    M(P) \geq \limsup_{n\to\infty} |\Delta(P_n)|^{\frac{1}{2nd}} \geq M(P)^{\frac{d-1}{2d}}.
\end{displaymath}
In particular, this implies the following:
\begin{thm} \label{equiv}
    The following are equivalent:
    \begin{itemize}
        \item There exists a constant $C > 1$ such that $M(P) \geq C$ for all noncyclotomic polynomials $P \in \mathbb{Z}[x]$.
        \item There exists a constant $C > 1$ such that $\limsup_{n\to\infty} |\Delta(P_n)|^{\frac{1}{nd}} \geq C$ for all noncyclotomic polynomials $P \in \mathbb{Z}[x]$ without multiple roots.
    \end{itemize}
\end{thm}

\subsection{Resultants}
The same method can be used to obtain a lower bound for the resultant $\Res(P_n,Q)$ for infinitely many $n$.
\begin{thm} \label{resultants}
    Let $P$ and $Q$ be two monic polynomials such that $Q(0) \neq 0$ and at least one of $P(1)$ and $Q(1)$ is nonzero. Then
    \begin{displaymath}
        \limsup_{n\to\infty} |\Res(P_n,Q)|^{\frac{1}{n}} = M(P)^{d_Q},
    \end{displaymath}
    where $d_Q$ is the degree of $Q$.
\end{thm}
\begin{proof}
    Let $\beta_1,\dots,\beta_{d_Q}$ be the roots of $Q$. We write the resultant as a product
    \begin{displaymath}
        \Res(P_n,Q) = \prod_{j=1}^d \prod_{k=1}^{d_Q} (\alpha_j^n-\beta_k)
    \end{displaymath}
    and note that the assumption that $P(1)$ and $Q(1)$ are not both zero implies that the resultant is nonzero for infinitely many $n$. Now consider the values of $j$ for which $|\alpha_j| > 1$. As $n$ tends to infinity, we have
    \begin{displaymath}
        \prod_{|\alpha_j| > 1} \prod_{k=1}^{d_Q} |\alpha_j^n-\beta_k| \sim M(P)^{nd_Q}
    \end{displaymath}
    because the individual factors are asymptotic to $|\alpha_j|^n$ as $n \to \infty$. For the remaining factors, we can use a method similar to the one used in the proof of Theorem \ref{lowerboundlimsup} to show that there exists a constant $c > 0$ such that
    \begin{displaymath}
        \prod_{|\alpha_j| \leq 1} \prod_{k=1}^{d_Q} |\alpha_j^n-\beta_k| \geq c
    \end{displaymath}
    for infinitely many $n$. Note that this constant $c$ depends on the absolute values of the roots of $Q$, so we need to assume that $Q(0) \neq 0$ to ensure that $c > 0$. We also have
    \begin{displaymath}
        \prod_{|\alpha_j| \leq 1} \prod_{k=1}^{d_Q} |\alpha_j^n-\beta_k| \leq \prod_{k=1}^{d_Q} (|\beta_k|+1)^{rd_Q},
    \end{displaymath}
    where $r$ is the number of roots of $P$ satisfying $|\alpha_j| \leq 1$. Taking the $n$-th root and the limit superior as $m \to \infty$, we see that
    \begin{displaymath}
        \limsup_{n\to\infty} |\Res(P_n,Q)|^{\frac{1}{n}} = M(P)^{d_Q}. \qedhere
    \end{displaymath}
\end{proof}

The following result due to Lehmer is a special case of this:
\begin{cor}[Lehmer \cite{Lehmer}]
    Let $\Delta_n := \prod_{j=1}^d (\alpha_j^n-1)$. Then
    \begin{displaymath}
        \limsup_{n\to\infty} |\Delta_n|^{\frac{1}{n}} = M(P).
    \end{displaymath}
\end{cor}
\begin{proof}
    Note that what Lehmer called $\Delta_n$ equals $\Res(P_n,x-1)$. Hence, this result is just a special case of Theorem \ref{resultants} with $Q(x) = x-1$.
\end{proof}

\section{The generating function of $\Delta(P_n)$ and related sequences}
Corollary \ref{upperboundlimsup} and Theorem \ref{equiv} have made it clear that studying the sequence $\{\Delta(P_n)\}$ may help solving Lehmer's problem. Consider the function
\begin{displaymath}
    f_P(z) := \sum_{n=1}^\infty \Delta(P_n)z^n,
\end{displaymath}
which is the generating function of the sequence $\{\Delta(P_n)\}$. We will write $f$ instead of $f_P$ when it is clear from the context which polynomial is meant. If $R$ is the radius of convergence of the above power series, then
\begin{displaymath}
    \limsup_{n\to\infty} |\Delta(P_n)|^{\frac{1}{n}} = \frac{1}{R},
\end{displaymath}
which implies that we can compute $\limsup_{n\to\infty} |\Delta(P_n)|^{\frac{1}{n}}$ if we know where the singularities of $f$ are.

It may be a bit easier to study the sequence $\{b_n\}$ given by
\begin{equation} \label{bn}
    b_n = \prod_{1 \leq j < k \leq d} \frac{\alpha_j^n-\alpha_k^n}{\alpha_j-\alpha_k},
\end{equation}
which satisfies $b_n^2 = \frac{\Delta(P_n)}{\Delta(P)}$. Clearly,
\begin{displaymath}
    \limsup_{n\to\infty} |\Delta(P_n)|^{\frac{1}{2nd}} = \limsup_{n\to\infty} |b_n|^{\frac{1}{nd}}.
\end{displaymath}
We set
\begin{displaymath}
    g_P(z) := \sum_{n=1}^\infty b_nz^n,
\end{displaymath}
where we again sometimes suppress $P$ from the notation.

\subsection{A simple example}
If $P(x) = x^2+a_1x+a_0$ is a quadratic polynomial, then $b_n$ is given by
\begin{displaymath}
    b_n = \sum_{j=0}^{n-1} \alpha_1^j\alpha_2^{n-1-j},
\end{displaymath}
and the sequence satisfies the recurrence relation
\begin{displaymath}
    b_n = (\alpha_1+\alpha_2)b_{n-1} - \alpha_1\alpha_2b_{n-2} = -a_1b_{n-1}-a_0b_{n-2}.
\end{displaymath}
This implies that
\begin{align*}
    g(z) &= z + \sum_{n=2}^\infty (-a_1b_{n-1}-a_0b_{n-2})z^n \\
    &= z -a_1zg(z) - a_0z^2g(z) \\
    &= \frac{z}{1+a_1z+a_0z^2}.
\end{align*}
We see that the singularities of $g(z)$ are located at $\alpha_1^{-1}$ and $\alpha_2^{-1}$. If we assume that $|\alpha_1| \geq |\alpha_2|$, this implies that
\begin{displaymath}
    \limsup_{n\to\infty} |b_n|^{\frac{1}{n}} = |\alpha_1|,
\end{displaymath}
which in turn implies that
\begin{displaymath}
    \limsup_{n\to\infty} |\Delta(P_n)|^{\frac{1}{2nd}} = |\alpha_1|^{\frac{1}{2}}.
\end{displaymath}
In this simple example, we can also determine $f$ explicitly. Since
\begin{displaymath}
    \Delta(P_n) = (\alpha_1^n-\alpha_2^n)^2 = \alpha_1^{2n}+\alpha_2^{2n}-2\alpha_1^n\alpha_2^n = \alpha_1^{2n}+\alpha_2^{2n}-2a_0^n,
\end{displaymath}
we have
\begin{displaymath}
    f(z) = \sum_{n=1}^\infty (\alpha_1^{2n}+\alpha_2^{2n}-2a_0^n)z^n = \frac{\alpha_1^2z}{1-\alpha_1^2z}+\frac{\alpha_2^2z}{1-\alpha_2^2z} - \frac{2a_0z}{1-a_0z},
\end{displaymath}
which has singularities at $\alpha_1^{-2}$, $\alpha_2^{-2}$ and $a_0^{-1}$. Furthermore, setting $\tilde{a}_1 = \alpha_1^2+\alpha_2^2$, we obtain
\begin{align*}
    f(z) &= -\frac{-\tilde{a}_1z + 2a_0^2z^2}{1-\tilde{a}_1z+a_0^2z^2} - 2 \frac{a_0z}{1-a_0z} \\
    &= -z\frac{(1-\tilde{a}_1z+a_0^2z^2)'}{1-\tilde{a}_1z+a_0^2z^2} + 2z \frac{(1-a_0z)'}{1-a_0z}.
\end{align*}
Note that this result is in accordance with the following theorem:
\begin{thm}[Minton \cite{Minton}]
    A rational function $f \in \mathbb{Q}(z)$ has the Gauss property (meaning that the coefficients of its Laurent series at 0 satisfy the Gauss congruences) if and only if $f$ is a $\mathbb{Q}$-linear combination of functions of the form $\frac{zu'(z)}{u(z)}$, where $u \in \mathbb{Z}[z]$.
\end{thm}

\subsection{Rationality and singularities}
In the quadratic example we studied above, we saw that both $f$ and $g$ are rational functions, and that their singularities are located at certain products of reciprocals of roots of $P$. Similar ideas can be used to show that this holds more generally.

\begin{thm} \label{fprational}
    Let $P \in \mathbb{Z}[x]$ be a monic polynomial without multiple roots. Then $f_P \in \mathbb{Q}(z)$, and all singularities of $f_P$ are simple poles occurring as monomials in the expression
    \begin{equation} \label{notquitedisc}
        \prod_{1\leq j < k \leq d} (\alpha_j^{-1}-\alpha_k^{-1})^2.
    \end{equation}
    Similarly, $g_P$ is a rational function whose singularities are simple poles occurring as monomials in the expression
    \begin{displaymath}
        \prod_{1\leq j < k \leq d} (\alpha_j^{-1}-\alpha_k^{-1}).
    \end{displaymath}
\end{thm}
\begin{proof}
    We only prove this for $f_P$, as the proof for $g_P$ is similar. First, we let $\beta_1,\dots,\beta_s$ be the products of reciprocals of roots of $P$ that appear when one expands \eqref{notquitedisc}, i.e.
    \begin{displaymath}
        \prod_{1\leq j < k \leq d} (\alpha_j^{-1}-\alpha_k^{-1})^2 = \sum_{j=1}^s c_j\beta_j
    \end{displaymath}
    for some integers $c_1,\dots,c_s$. Expanding the product definition of the discriminant, we get the same as above but with $\alpha_j^{-1}$ replaced with $\alpha_j^n$. Hence
    \begin{displaymath}
        \Delta(P_n) = \sum_{j=1}^s c_j\beta_j^{-n}.
    \end{displaymath}
    Then
    \begin{equation} \label{generating}
        f_P(z) = \sum_{n=1}^\infty \sum_{j=1}^s c_j\beta_j^{-n}z^n = \sum_{j=1}^s \frac{c_j\beta_j^{-1}z}{1-\beta_j^{-1}z},
    \end{equation}
    which is clearly a rational function with simple poles at $\beta_j$ and no other singularities.
\end{proof}

Combined with Minton's theorem, this results in the following:
\begin{thm} \label{Minton}
    Let $P \in \mathbb{Z}[x]$ be a monic polynomial. Then there exist polynomials $u_1,\dots,u_r \in \mathbb{Z}[z]$ and rational numbers $c_1,\dots,c_r$ such that
    \begin{displaymath}
        f_P(z) = \sum_{j=1}^r c_j\frac{zu_j'(z)}{u_j(z)}.
    \end{displaymath}
\end{thm}

\begin{rmk}
    This result can also be derived directly from \eqref{generating} as the constant $c_j$ is the same for all algebraic conjugates of $\beta_j$. Since all $c_j$ are integers, we see that the rational numbers $c_1,\dots,c_r$ in Theorem \ref{Minton} are actually integers.

    Theorem \ref{Minton} does not hold if we replace $f_P$ by $g_P$ because its coefficients do not satisfy the Gauss congruences.
\end{rmk}

\section{Factorisation of $\Delta(P_n)$ and $U(n)$}
A close look at the factors of the products defining $\Delta(P_n)$ and $U(n)$ reveals that both of these integers can be factored.

\subsection{Factors of $\Delta(P_n)$}
Let $\Phi_n$ denote the $n$-th cyclotomic polynomial. Then from
\begin{displaymath}
    \prod_{m \mid n} \Phi_m(x) = x^n-1
\end{displaymath}
we have
\begin{displaymath}
    (\alpha_j^n-\alpha_k^n) = \alpha_k^n\Big(\frac{\alpha_j^n}{\alpha_k^n}-1\Big) = \alpha_k^n \prod_{m \mid n} \Phi_m\Big(\frac{\alpha_j}{\alpha_k}\Big) = \prod_{m \mid n} \alpha_k^{\phi(m)} \Phi_m\Big(\frac{\alpha_j}{\alpha_k}\Big).
\end{displaymath}
If we now set
\begin{displaymath}
    \Psi_n(P) := \prod_{j < k} \bigg(\alpha_k^{\phi(n)} \Phi_n\Big(\frac{\alpha_j}{\alpha_k}\Big)\bigg)^2,
\end{displaymath}
we obtain
\begin{equation} \label{Discproduct}
    \Delta(P_n) = \prod_{m \mid n} \Psi_m(P),
\end{equation}
where, as usual, we have assumed that $P$ is monic. Note the similarity between this and Lehmer's work on factorising $\prod (\alpha_j^n-1)$. In particular, if $|a_0| = 1$, then all numbers of the form $\alpha_j/\alpha_k$ are algebraic integers, so that Lehmer's results can be applied to $\Delta(P_n)$. Observe that in this case we have
\begin{displaymath}
    \Psi_n(P) = \prod_{j < k} \bigg(\Phi_n\Big(\frac{\alpha_j}{\alpha_k}\Big)\bigg)^2.
\end{displaymath}

Following Lehmer, we call $\Psi_n(P)$ the essential factor of $\Delta(P_n)$ and we define the characteristic prime factors of $\Delta(P_n)$ to be the primes $p \mid \Delta(P_n)$ that do not divide $\Delta(P_m)$ when $m$ is a proper divisor of $n$. For example, if we take $P(x) = x^4+x^3-x^2+x+1$, then $\Delta(P) = -3\cdot13^2$ and $\Delta(P_5) = -2^{16}\cdot3\cdot13^2$, so $2$ is a characteristic prime factor of $\Delta(P_5)$.

\begin{thm}
    Let $p$ be a characteristic prime factor of $\Delta(P_n)$. Then $p$ divides the essential factor $\Psi_n(P)$. Moreover, if $p^k$ is the highest power of $p$ dividing $\Delta(P_n)$, then $p^k \equiv 1 \mod n$.
\end{thm}
\begin{proof}
    This follows from Theorems 1, 2 and 3 of \cite{Lehmer}.
\end{proof}

\begin{thm} \label{Psisquare}
    For $n > 1$, $\Psi_n(P)$ is a square.
\end{thm}
\begin{proof}
    Since the cyclotomic polynomial $\Phi_n$ is reciprocal for $n>1$, the expression
    \begin{displaymath}
        \alpha_k^{\phi(n)}\Phi_n\Big(\frac{\alpha_j}{\alpha_k}\Big)
    \end{displaymath}
    is symmetric in $\alpha_j$ and $\alpha_k$. Hence
    \begin{displaymath}
        \prod_{j < k} \alpha_k^{\phi(n)} \Phi_n\Big(\frac{\alpha_j}{\alpha_k}\Big)
    \end{displaymath}
    is an integer, so $\Psi_n(P)$ is a square.
\end{proof}

\begin{cor}
    For all $n$, the discriminant $\Delta(P_n)$ has the same sign.
\end{cor}
\begin{proof}
    By Theorem \ref{Psisquare}, all factors on the right-hand side of \eqref{Discproduct} except possibly $\Psi_1(P)$ are positive. Hence, $\Delta(P_n)$ has the same sign as $\Psi_1(P) = \Delta(P)$ for all $n$.
\end{proof}

We also have the following divisibility result.
\begin{thm} \label{Discdiv}
    Let $P \in \mathbb{Z}[x]$ be a monic polynomial with constant coefficient $a_0 = \pm 1$, let $K = \mathbb{Q}(\alpha_1,\dots,\alpha_d)$ be its splitting field, let $R = \mathcal{O}_K$ be its ring of integers, and let $m \in \mathbb{Z}_{>1}$. If $n$ is divisible by $\#(R/mR)^\times$, then $\Delta(P_n)$ is divisible by $m^{d(d-1)}$.
\end{thm}
\begin{proof}
    Since $a_0 = \pm 1$, all $\alpha_j$ are units in $R$. Hence, their residue classes are units in the quotient $R/mR$. This implies that for all $j$ we have
    \begin{displaymath}
        \alpha_j^n \equiv 1 \bmod mR,
    \end{displaymath}
    which implies that $\alpha_j^n-\alpha_k^n \in mR$ for all $j,k$. Hence, $\Delta(P_n) \in m^{d(d-1)}R$. Since $\Delta(P_n) \in \mathbb{Z}$, the statement follows.
\end{proof}

\begin{ex}
    Let $P(x) = x^2-3x+1$ and $m = 2$. Then $R = \mathbb{Z}[\alpha_j]$, where $\alpha_j$ is either root of $P$. Since $P$ is irreducible modulo 2, we have $(R/2R)^\times \cong \mathbb{F}_4^\times$, which has 3 elements. Thus, according to Theorem \ref{Discdiv}, $\Delta(P_3)$ should be divisible by $2^2$. Indeed, we have $\Delta(P_3) = 2^6\cdot 5$, with even stronger divisibility than predicted by the theorem.
\end{ex}

In practice, the condition in by Theorem \ref{Discdiv} is often stronger than necessary: in some cases, the order of $\alpha$ in the ring $(R/mR)^\times$ is strictly less than $\#(R/mR)^\times$. When this is the case, divisibility by $m^{d(d-1)}$ occurs for lower values of $n$. In particular, the following holds.
\begin{thm} \label{Discdiv2}
    Let $P \in \mathbb{Z}[x]$ be a monic polynomial with constant coefficient $a_0 = \pm 1$, let $K$ be its splitting field, let $R$ be the ring of integers of $K$, and let $m \in \mathbb{Z}_{>1}$. For each root $\alpha_j$ of $P$, let $m_j$ denote the order of $\alpha_j$ in the group $(R/mR)^\times$, and let $M = \lcm(m_1,\dots,m_d)$. If $n$ is divisible by $M$, then $\Delta(P_n)$ is divisible by $m^{d(d-1)}$.
\end{thm}
An additional advantage of Theorem \ref{Discdiv2} over Theorem \ref{Discdiv} is that it does not require computing $R$: $m_j$ can be computed without knowledge of $R$.

\begin{ex}
    Let $P(x) = x^3-x-1$ and $m=2$. Since $P$ is irreducible modulo 2, for each $\alpha_j$ we have $\mathbb{Z}[\alpha_j]/2\mathbb{Z}[\alpha_j] \cong \mathbb{F}_2(\alpha_j) \cong \mathbb{F}_8$. Hence, $m_j = 7$ for all $j$, so according to Theorem \ref{Discdiv2}, $\Delta(P_7)$ should be divisible by $2^6$. Indeed, $\Delta(P_7) = -2^6\cdot23$.
\end{ex}

\begin{ex}
    Let $P(x) = x^4+x^3-x^2+x+1$ and $m=2$. $P$ is irreducible modulo 2, so $m_j$ is the same for all $j$, and $m_j \mid 15$. Moreover, since $P$ is reciprocal and the Galois group $\operatorname{Gal}(\mathbb{F}_{16}/\mathbb{F}_2)$ is generated by the Frobenius automorphism, $\alpha^{-1} = \alpha^{2^k}$ for some $k$. Since we also have $\alpha = \alpha^{-2^k} = (\alpha^{2^k})^{2^k}$, it follows that $4^k = 16$, so $2^k = 1$. Hence, $\alpha^5 = 1$, so $m_j = 5$. This implies that $\Delta(P_5)$ is divisible by $2^{12}$.
\end{ex}

\subsection{A sequence related to $\Delta(P_n)$}
The Gauss congruences for $\Delta(P_n)$ suggest defining the following sequence of integers:
\begin{displaymath}
    \delta_n(P) := \sum_{m \mid n} \mu\Big(\frac nm\Big) \Delta(P_m).
\end{displaymath}
By the Gauss congruences, we have $n \mid \delta_n(P)$ for all $n$, and by Möbius inversion we also have
\begin{equation} \label{Deltasum}
    \Delta(P_n) = \sum_{m \mid n} \delta_m(P).
\end{equation}

The sequence $\{\delta_n(P)\}$ has the following property:
\begin{thm} \label{smalldeltasign}
    Let $P \in \mathbb{Z}[x]$ be a monic polynomial such that $\Delta(P_n) \neq 0$ for all $n$. Then the integer $\delta_n(P)$ has the same sign for all $n$.
\end{thm}

To prove this theorem, we need the following lemma.
\begin{lem} \label{sequences}
    Let $\{a_n\}$ be a sequence of positive numbers with the property that for all $n, k$, the inequality
    \begin{displaymath}
        \sum_{m \mid n} \mu\Big(\frac nm\Big) a_{mk} \geq 0
    \end{displaymath}
    holds, and let $N \in \mathbb{Z}_{\geq 1}$ and $C \geq 1$. Define the sequence
    \begin{displaymath}
        b_n =
        \begin{cases}
            Ca_n & \text{if } N \mid n \\
            a_n & \text{if } N \nmid n.
        \end{cases}
    \end{displaymath}
    Then the sequence $\{b_n\}$ satisfies the inequality
    \begin{displaymath}
        \sum_{m \mid n} \mu\Big(\frac nm\Big) b_{mk} \geq 0
    \end{displaymath}
    for all $n, k$.
\end{lem}
\begin{proof}
    If $N \nmid nk$, then $N \nmid km$ for all $m \mid n$, so
    \begin{displaymath}
        \sum_{m \mid n} \mu\Big(\frac nm\Big) b_{mk} = \sum_{m\mid n} \Big(\frac nm \Big) a_{mk} \geq 0.
    \end{displaymath}
    If $N \mid nk$, then
    \begin{displaymath}
        \sum_{m \mid n} \mu\Big( \frac nm\Big) b_{km} = (C-1) \sum_{\substack{m \mid n \\ N \mid mk}} \mu\Big( \frac nm\Big) a_{mk} + \sum_{m \mid n} \mu\Big(\frac nm\Big) a_{mk}.
    \end{displaymath}
    By assumption, $C-1 \geq 0$ and $\sum \mu(n/m) a_{mk} \geq 0$, so we are done if we can show that
    \begin{displaymath}
        \sum_{\substack{m \mid n \\ N \mid mk}} \mu\Big( \frac nm\Big) a_{mk} \geq 0.
    \end{displaymath}
    Setting $m' = mk$, we can rewrite this sum as follows:
    \begin{displaymath}
        \sum_{\substack{m \mid n \\ N \mid mk}} \mu\Big( \frac nm\Big) a_{mk} = \sum_{\substack{m' \mid nk \\ N \mid m' \\ k \mid m'}} \mu\Big(\frac{nk}{m'}\Big) a_{m'} = \sum_{\substack{m' \mid nk \\ \lcm(N,k) \mid m'}} \mu\Big(\frac{nk}{m'}\Big) a_{m'}.
    \end{displaymath}
    If we now set $m'' = \frac{m'}{\lcm(N,k)}$, we get
    \begin{displaymath}
        \sum_{\substack{m' \mid nk \\ \lcm(N,k) \mid m'}} \mu\Big(\frac{nk}{m'}\Big) a_{m'} = \sum_{\substack{m''|\frac{nk}{\lcm(N,k)}}} \mu\Big(\frac{nk}{\lcm(N,k)m''}\Big) a_{\lcm(N,k)m''} \geq 0,
    \end{displaymath}
    where the inequality follows from the assumption on the sequence $\{a_n\}$.
\end{proof}

\begin{proof}[Proof of Theorem \ref{smalldeltasign}]
    Since $\Delta(P_n)$ has the same sign for all $n$, we may without loss of generality assume that $\Delta(P_n) > 0$ for all $n$, so that the statement we want to prove is that $\delta_n(P) \geq 0$ for all $n$. We prove this by recursively constructing a bi-indexed sequence $\{a_{l,n}\}$ as follows: let $a_{1,n} = \Delta(P)$ and for $l \geq 2$, define $a_{l,n}$ by
    \begin{displaymath}
        a_{l,n} =
        \begin{cases}
            \Psi_l(P)a_{l-1,n} & \text{if } l\mid n \\
            a_{l-1,n} & \text{if } l \nmid n.
        \end{cases}
    \end{displaymath}
    Since $\{a_{1,n}\}$ is a constant sequence, it certainly satisfies the condition of Lemma \ref{sequences}. Suppose that for some $l$, the inequality
    \begin{displaymath}
        \sum_{m \mid n} \mu\Big(\frac nm\Big) a_{l-1,mk} \geq 0
    \end{displaymath}
    holds for all $n$ and $k$. Then Lemma \ref{sequences} and the fact that $\Psi_l(P) \geq 1$ together imply that
    \begin{displaymath}
        \sum_{m \mid n} \mu\Big(\frac nm\Big) a_{l,mk} \geq 0
    \end{displaymath}
    holds for all $n$ and $k$. By induction on $l$, it follows that this inequality holds for all $l$, $n$ and $k$. By construction, $a_{l,n} = \Delta(P_n)$ for $n \leq l$ (see \eqref{Discproduct}), so this implies that
    \begin{displaymath}
        \delta_n(P) = \sum_{m \mid n} \mu\Big(\frac nm\Big) \Delta(P_m) = \sum_{m \mid n} \mu \Big( \frac nm\Big) a_{n,m} \geq 0. \qedhere
    \end{displaymath}
\end{proof}

\begin{rmk}
    The condition $\Delta(P_n) \neq 0$ for all $n$ is necessary: if $\Delta(P_n) = 0$ for some $n$, then the right-hand side of \eqref{Deltasum} must contain both positive and negative terms.
\end{rmk}

\begin{cor}
    Let $P \in \mathbb{Z}[x]$ be a monic polynomial such that $\Delta(P_n) \neq 0$ for all $n$. If $p \nmid n$ is a prime and $k$ is a positive integer, then
    \begin{displaymath}
        |\delta_n(P_{p^k})| \geq |\delta_n(P_{p^{k-1}})|.
    \end{displaymath}
\end{cor}
\begin{proof}
    Without loss of generality, we assume $\Delta(P) > 0$. Then
    \begin{align*}
        \delta_{np^k}(P) &= \sum_{m \mid np^k} \mu\Big(\frac{np^k}{m}\Big) \Delta(P_m) \\
        &= \sum_{m\mid n} \mu\Big(\frac{np^k}{mp^k}\Big) \Delta(P_{mp^k}) + \mu\Big(\frac{np^k}{mp^{k-1}}\Big) \Delta(P_{mp^{k-1}}) \\
        &= \sum_{m \mid n} \mu\Big(\frac nm\Big) \big(\Delta(P_{mp^k}) - \Delta(P_{mp^{k-1}})\big) \\
        &= \delta_n(P_{p^k}) - \delta_n(P_{p^{k-1}}).
    \end{align*}
    Since $\delta_{np^k}(P) \geq 0$, this implies that $\delta_n(P_{p^k}) \geq \delta_n(P_{p^{k-1}})$.
\end{proof}

Another corollary of Theorem \ref{smalldeltasign} is the following:
\begin{cor}
    We have
    \begin{displaymath}
        |\Delta(P_n)| = \sum_{m \mid n} |\delta_m(P)|.
    \end{displaymath}
\end{cor}
In particular, this implies that a lower bound for $|\delta_m(P)|$ immediately implies a lower bound for $|\Delta(P_n)|$.

\begin{thm}
    For sufficiently large $n$, $\delta_n(P)$ is divisible by $12$.
\end{thm}
\begin{proof}
    We prove this by showing that $\delta_n(P)$ is divisible by 3 and 4 for sufficiently large $n$. Since the arguments are similar, we only give details for divisibility by 3. What we will actually show is that
    \begin{equation} \label{divby3}
        \sum_{m\mid n} \mu\Big(\frac{n}{m}\Big) \frac{\Delta(P_m)}{\Delta(P)}
    \end{equation}
    is divisible by 3 for all but finitely many $n$. Note that $\Delta(P_m)/\Delta(P)$ is always a square, so it is congruent to either 0 or $1 \bmod 3$. Hence, we are done if we can show that $1 \bmod 3$ appears equally often with a plus sign as with a minus sign.
    
    For each pair $(j,k)$ with $j < k$, let $m_{jk}$ be the smallest positive integer such that $\alpha_j^{m_{jk}}-\alpha_k^{m_{jk}} \equiv 0 \mod 3$. Then clearly, $\Delta(P_m)$ is divisible by 3 if and only if $m_{jk} \mid m$ for some $m_{jk}$. Now let $n_0$ be the lowest common multiple of all $m_{jk}$ and let $n>n_0$. Let $n = p_1^{k_1}\dots p_l^{k_l}$ be the prime factorisation of $n$, and let $n' = p_1^{k_1-1}\dots p_l^{k_l-1}$ be the greatest common divisor of all $m$. If $m_{jk} \mid n'$ for some $m_{jk}$, then $m_{jk} \mid m$ for all $m$, which implies that $\Delta(P_m)$ is divisible by 3 for all $m$, so we assume that $n'$ is not divisible by any $m_{jk}$. Define the sets
    \begin{align*}
        A &= \{ m_{jk}: m_{jk} \mid n\}, \\
        B_1 &= \{ m \mid n: \mu(n/m) = 1 \text{ and } 3 \nmid \Delta(P_m)/\Delta(P)\} \\
        &= \{ m \mid n: \mu(n/m) = 1 \text{ and } \forall m_{jk} \in A: m_{jk} \nmid m\}, \\
        B_2 &= \{ m \mid n: \mu(n/m) = -1 \text{ and } 3 \nmid \Delta(P_m)/\Delta(P)\}.
    \end{align*}
    Then the assumption that $n > n_0$ implies that $\lcm(m_{jk} \in A) < n$, so for some prime $p$, all elements of $A$ divide $\frac np$. This implies that the following map from $B_1$ to $B_2$ is a bijection:
    \begin{displaymath}
        m \mapsto
        \begin{cases}
            mp & p \nmid \frac nm \\
            \frac mp & p \mid \frac nm.
        \end{cases}
    \end{displaymath}
    It follows that $B_1$ and $B_2$ have the same number of elements, so their contributions to \eqref{divby3} cancel out. Hence, $\delta_n(P)$ is divisibly by 3. Since this holds for all $n > n_0$, $\delta_n(P)$ is divisible by 3 for all but finitely many $n$. A similar argument shows that $\delta_n(P)$ is also divisible by 4 for all but finitely many $n$, so the theorem follows.
\end{proof}

The divisibility of $\delta_n(P)$ by $n$ and Theorem \ref{Minton} have the following corollary.
\begin{thm}
    Let $P \in \mathbb{Z}[x]$ be a monic polynomial without multiple roots, let $r_n = \frac 1n \delta_n(P)$, and let $u_j, c_j$ be as in Theorem \ref{Minton}. Then
    \begin{displaymath}
        \prod_{n=1}^\infty (1-z^n)^{r_n} = \prod_{j=1}^r u_j(z)^{-c_j}.
    \end{displaymath}
\end{thm}
\begin{proof}
    From Theorem \ref{Minton} and the following remark, we know that
    \begin{displaymath}
        f_P(z) = \sum_{n=1}^\infty \Delta(P_n)z^n = \sum_{j=1}^r c_j\frac{zu_j'(z)}{u_j(z)},
    \end{displaymath}
    where $u_j \in \mathbb{Z}[z]$ and $c_j \in \mathbb{Z}$. By \eqref{Deltasum}, $f_P$ can also be written as
    \begin{displaymath}
        f_P(z) = \sum_{n=1}^\infty \frac{\delta_n(P)z^n}{1-z^n} = \sum_{n=1}^\infty \frac{nr_nz^n}{1-z^n}.
    \end{displaymath}
    Observe that both of these expressions for $f_P$ involve a logarithmic derivative:
    \begin{displaymath}
        f_P(z) = z\frac{d}{dz} \log \prod_{j=1}^r u_j(z)^{c_j} = -z\frac{d}{dz} \log \prod_{n=1}^\infty (1-z^n)^{r_n}.
    \end{displaymath}
    This implies that
    \begin{displaymath}
        \prod_{n=1}^\infty (1-z^n)^{r_n} = C\prod_{j=1}^r u_j(z)^{-c_j}
    \end{displaymath}
    for some constant $C$. Since each $u_j$ is a product of factors of the form $(1-\beta^{-1}z)$, the constant term of $\prod_{j=1}^r u_j(z)^{-c_j}$ equals 1, so $C=1$.
\end{proof}

\subsection{Factors of $U(n)$}
One can easily identify a lot of integer factors of $U(n)$. Recall that $U(n)$ is given by
\begin{displaymath}
    U(n) = \prod_{j_s=1}^d \Bigg( \prod_{m_s \in \mathcal{D}_+(n)} \alpha_{j_s}^{m_s} - \prod_{m_s \in \mathcal{D}_-(n)} \alpha_{j_s}^{m_s}\Bigg),
\end{displaymath}
where we have a separate index $j_s$ for each $m \mid n$ satisfying $\mu(n/m) \neq 0$. Note that $s$ runs from $1$ to $2^l$, where $l$ is the number of distinct prime factors of $n$. We obtain integer factors of $U(n)$ by requiring some of the indices to be equal. More specifically, let $\mathcal{P}$ be a partition of the set $\{1,\dots,2^l\}$ into at most $d$ blocks, and let
\begin{displaymath}
    U(n, \mathcal{P}) := \sideset{}{^*}\prod_{j_s=1}^d \Bigg( \prod_{m_s \in \mathcal{D}_+(n)} \alpha_{j_s}^{m_s} - \prod_{m_s \in \mathcal{D}_-(n)} \alpha_{j_s}^{m_s}\Bigg),
\end{displaymath}
where the product again has a separate index $j_s$ for $s \in \{1,\dots,2^l\}$, and we only include the factors that satisfy $j_s = j_{s'}$ if and only if $s$ and $s'$ are in the same block of $\mathcal{P}$. Then $U(n, \mathcal{P})$ is an integer, and
\begin{displaymath}
    U(n) = \prod_{\mathcal{P}} U(n, \mathcal{P}).
\end{displaymath}

This complicated definition is probably best clarified by giving some concrete examples. First, if $n=p^k$ is a prime power, we have the two partitions $\mathcal{P}_1 = \{\{1,2\}\}$ and $\mathcal{P}_2 = \{\{1\},\{2\}\}$ giving rise to the factors
\begin{align*}
    U(p^k, \mathcal{P}_1) &= \prod_{j=1}^d (\alpha_j^{p^k}-\alpha_j^{p^{k-1}}), \\
    U(p^k, \mathcal{P}_2) &= \prod_{\substack{j_1,j_2=1 \\ j_1 \neq j_2}}^d (\alpha_{j_1}^{p^k}-\alpha_{j_2}^{p^{k-1}}).
\end{align*}

If $n=pq$ is the product of two primes, we have
\begin{displaymath}
    U(n) = \prod_{j_1,j_2,j_3,j_4=1}^d (\alpha_{j_1}^n\alpha_{j_2}-\alpha_{j_3}^p\alpha_{j_4}^q).
\end{displaymath}
Assuming $d \geq 4$, the above method gives rise to 15 factors of $U(n)$, corresponding to the 15 partitions of the set $\{1,2,3,4\}$. We have the following partitions:
\begin{align*}
    \mathcal{P}_1 &= \{\{1, 2, 3, 4\}\}, \\
    \mathcal{P}_2 &= \{\{1, 2, 3\}, \{4\}\}, \\
    \mathcal{P}_3 &= \{\{1, 2, 4\}, \{3\}\}, \\
    \mathcal{P}_4 &= \{\{1, 3, 4\}, \{2\}\}, \\
    \mathcal{P}_5 &= \{\{2, 3, 4\}, \{1\}\}, \\
    \mathcal{P}_6 &= \{\{1, 2\}, \{3, 4\}\}, \\
    \mathcal{P}_7 &= \{\{1, 2\}, \{3\}, \{4\}\}, \\
    \mathcal{P}_8 &= \{\{1\}, \{2\}, \{3, 4\}\}, \\
    \mathcal{P}_9 &= \{\{1, 3\}, \{2, 4\}\}, \\
    \mathcal{P}_{10} &= \{\{1, 3\}, \{2\}, \{4\}\}, \\
    \mathcal{P}_{11} &= \{\{1\}, \{3\}, \{2, 4\}\}, \\
    \mathcal{P}_{12} &= \{\{1, 4\}, \{2, 3\}\}, \\
    \mathcal{P}_{13} &= \{\{1, 4\}, \{2\}, \{3\}\}, \\
    \mathcal{P}_{14} &= \{\{1\}, \{4\}, \{2, 3\}\}, \\
    \mathcal{P}_{15} &= \{\{1\}, \{2\}, \{3\}, \{4\}\},
\end{align*}
corresponding to the factors
\begin{align*}
    U(n, \mathcal{P}_1) &= \prod_{\substack{j_1,j_2,j_3,j_4=1 \\ j_1=j_2=j_3=j_4}}^d (\alpha_{j_1}^n\alpha_{j_2}-\alpha_{j_3}^p\alpha_{j_4}^q), \qquad
    U(n, \mathcal{P}_2) = \prod_{\substack{j_1,j_2,j_3,j_4=1 \\ j_1=j_2=j_3\neq j_4}}^d (\alpha_{j_1}^n\alpha_{j_2}-\alpha_{j_3}^p\alpha_{j_4}^q), \\
    U(n, \mathcal{P}_3) &= \prod_{\substack{j_1,j_2,j_3,j_4=1 \\ j_1=j_2=j_4\neq j_3}}^d (\alpha_{j_1}^n\alpha_{j_2}-\alpha_{j_3}^p\alpha_{j_4}^q), \qquad
    U(n, \mathcal{P}_4) = \prod_{\substack{j_1,j_2,j_3,j_4=1 \\ j_1=j_3=j_4\neq j_2}}^d (\alpha_{j_1}^n\alpha_{j_2}-\alpha_{j_3}^p\alpha_{j_4}^q), \displaybreak[2]\\
    U(n, \mathcal{P}_5) &= \prod_{\substack{j_1,j_2,j_3,j_4=1 \\ j_2=j_3=j_4\neq j_1}}^d (\alpha_{j_1}^n\alpha_{j_2}-\alpha_{j_3}^p\alpha_{j_4}^q), \qquad
    U(n, \mathcal{P}_6) = \prod_{\substack{j_1,j_2,j_3,j_4=1 \\ j_1=j_2\neq j_3=j_4}}^d (\alpha_{j_1}^n\alpha_{j_2}-\alpha_{j_3}^p\alpha_{j_4}^q), \displaybreak[2]\\
    U(n, \mathcal{P}_7) &= \prod_{\substack{j_1,j_2,j_3,j_4=1 \\ j_1=j_2 \neq j_3 \neq j_4}}^d (\alpha_{j_1}^n\alpha_{j_2}-\alpha_{j_3}^p\alpha_{j_4}^q), \qquad
    U(n, \mathcal{P}_8) = \prod_{\substack{j_1,j_2,j_3,j_4=1 \\ j_1 \neq j_2 \neq j_3 = j_4}}^d (\alpha_{j_1}^n\alpha_{j_2}-\alpha_{j_3}^p\alpha_{j_4}^q), \displaybreak[2]\\
    U(n, \mathcal{P}_9) &= \prod_{\substack{j_1,j_2,j_3,j_4=1 \\ j_1=j_3\neq j_2=j_4}}^d (\alpha_{j_1}^n\alpha_{j_2}-\alpha_{j_3}^p\alpha_{j_4}^q), \qquad
    U(n, \mathcal{P}_{10}) = \prod_{\substack{j_1,j_2,j_3,j_4=1 \\ j_1=j_3\neq j_2 \neq j_4}}^d (\alpha_{j_1}^n\alpha_{j_2}-\alpha_{j_3}^p\alpha_{j_4}^q), \displaybreak[2]\\
    U(n, \mathcal{P}_{11}) &= \prod_{\substack{j_1,j_2,j_3,j_4=1 \\ j_1 \neq j_3 \neq j_2 = j_4}}^d (\alpha_{j_1}^n\alpha_{j_2}-\alpha_{j_3}^p\alpha_{j_4}^q), \qquad   
    U(n, \mathcal{P}_{12}) = \prod_{\substack{j_1,j_2,j_3,j_4=1 \\ j_1=j_4\neq j_2=j_3}}^d (\alpha_{j_1}^n\alpha_{j_2}-\alpha_{j_3}^p\alpha_{j_4}^q), \displaybreak[2]\\
    U(n, \mathcal{P}_{13}) &= \prod_{\substack{j_1,j_2,j_3,j_4=1 \\ j_1=j_4 \neq j_2 \neq j_3}}^d (\alpha_{j_1}^n\alpha_{j_2}-\alpha_{j_3}^p\alpha_{j_4}^q), \qquad
    U(n, \mathcal{P}_{14}) = \prod_{\substack{j_1,j_2,j_3,j_4=1 \\ j_1 \neq j_4 \neq j_2 = j_3}}^d (\alpha_{j_1}^n\alpha_{j_2}-\alpha_{j_3}^p\alpha_{j_4}^q), \\
    U(n, \mathcal{P}_{15}) &= \prod_{\substack{j_1,j_2,j_3,j_4=1 \\ j_1 \neq j_2 \neq j_3 \neq j_4}}^d (\alpha_{j_1}^n\alpha_{j_2}-\alpha_{j_3}^p\alpha_{j_4}^q),
\end{align*}
where expressions like $j_1 \neq j_2 \neq j_3$ should be interpreted as saying that $j_1 \neq j_3$ as well. With the above notation, we have
\begin{displaymath}
    U(n) = \prod_{j=1}^{15} U(n, \mathcal{P}_j).
\end{displaymath}
The number of factors of $U(n)$ increases very quickly as the number of distinct prime factors increases: when $n$ has three distinct prime factors, $U(n)$ already has so many factors that it is not practical to write them all out.

If $P$ is reciprocal, we can split these into even more factors, because we can additionally require that some of the roots of $P$ appearing in the product be inverses of each other. For instance, in the above example with $n=pq$, $U(n, \mathcal{P}_2)$ has the factors
\begin{align*}
    \prod_{j=1}^d (\alpha_j^n\alpha_j-\alpha_j^p\alpha_j^{-q}) &= \prod_{j=1}^d (\alpha_j^{n+1}-\alpha_j^{p-q}), \\
    \prod_{\substack{j,k=1 \\ \alpha_j^{-1} \neq \alpha_k}}^d (\alpha_j^n\alpha_j -\alpha_j^p\alpha_k^q) &= \prod_{\substack{j,k=1 \\ \alpha_j^{-1} \neq \alpha_k}}^d (\alpha_j^{n+1}-\alpha_j^p\alpha_k^q).
\end{align*}

\begin{thm} \label{Undiv}
    Let $P$ be a polynomial with constant coefficient $\pm 1$, let $k$ be an integer and let $n$ be an integer with the property that $\phi(n)$ is divisible by $\#(R/kR)^\times$. Then $U(n)$ is divisible by $k^d$.
\end{thm}
\begin{proof}
    Consider the following factor of $U(n)$:
    \begin{displaymath}
        \prod_{j=1}^d \Big(\prod_{m\in\mathcal{D}_+(n)} \alpha_j^m-\prod_{m\in\mathcal{D}_-(n)} \alpha_j^m\Big) = \pm a_0^{\sum_{m\in\mathcal{D}_-(n)}m} \prod_{j=1}^d (\alpha_j^{\phi(n)}-1).
    \end{displaymath}
    By assumption, $\alpha_j^{\phi(n)} \equiv 1 \bmod k$ for all $j$, so each factor of this product is divisible by $k$. Hence, $U(n)$ is divisible by $k^d$.
\end{proof}
\begin{rmk}
    By Dirichlet's theorem on arithmetic progressions, there are infinitely many primes $p$ for which $\phi(p)$ is divisible by $\#(R/kR)^\times$. Therefore, there are infinitely many primes (and hence infinitely many integers) $n$ for which $U(n)$ is divisible by $k^d$.
\end{rmk}

\begin{ex}
    Let $P(x) = x^2-3x+1$ and $k=2$. As we saw before, $\#(R/kR)^\times = 3$, so if $3 \mid \phi(n)$, $U(n)$ is divisible by $2^2$. For example, if we take $n=7$, we get $U(n) = 2^6\cdot3^2\cdot5^2\cdot7^2$, which is indeed divisible by $2^2$.
\end{ex}

Similarly to Theorem \ref{Discdiv}, the condition in Theorem \ref{Undiv} is often stronger than needed. For example, let us consider what happens when $P$ is any reciprocal quadratic polynomial.

First, we examine divisibility by 2. If $P$ is reducible modulo 2, then $P \equiv (x+1)^2 \bmod 2$, so both roots of $P$ satisfy $\alpha \equiv 1 \bmod 2$. Hence, all factors of $U(n)$ are divisible by 2, which implies that $U(n)$ is divisible by $2^{2^{2^l}}$. If $P$ is irreducible modulo 2, then $\alpha^2 \equiv \alpha^{-1} \bmod 2$ and $\alpha^3 \equiv 1 \bmod 2$. If $n = p^k$ is a prime power with $p \neq 3$, then
\begin{displaymath}
    U(n) = \prod_{j_1,j_2=1}^2 (\alpha_{j_1}^{p^k}-\alpha_{j_2}^{p^{k-1}}).
\end{displaymath}
Now either $p^k \equiv p^{k-1} \equiv 1 \bmod 3$, or $p^k \equiv - p^{k-1} \mod 3$. In both cases, we see that the above product contains two factors that are divisible by 2, so $U(n)$ is divisible by $2^2$. If $p = 3$ and $k > 1$, all factors are divisible by 2, so $U(n)$ is divisible by $2^4$.

Next we consider divisibility of $U(p)$ by 3. If $P$ is irreducible modulo 3, then $P(x) \equiv x^2+1 \bmod 3$, which implies that the roots have order 4 in $(R/3R)^\times$, and $\alpha_1 \equiv -\alpha_2 \bmod 3$. This implies that for odd $p$, $U(p)$ has two factors which are divisible by 3, so $U(p)$ is divisible by $3^2$.

If $P$ is irreducible modulo 3, then $\alpha_1 \equiv \alpha_2 \equiv \pm 1 \bmod 3$, which implies that for odd $p$, $U(p)$ is divisible by $3^4$.

A similar argument shows that $U(p)$ is divisible by 5 for $p > 2$. These arguments cannot be directly applied to higher primes or polynomials of higher degree, and indeed it turns out that $U(n)$ is not in general divisible by $p$.

\begin{ex}
    Let $P(x) = x^2-3x+1$. Since this a quadratic reciprocal polynomial, the above results imply that $U(p)$ is divisible by 2, 3 and 5. However, $U(p)$ is not in general divisible by 7: by looking at the orders of the roots of $P$ in $(R/7R)^\times$, one can show that $U(p)$ is divisible by 7 if and only if $p \equiv \pm 1 \bmod 8$.
\end{ex}

\begin{ex}
    Let $P(x) = x^4+x^3-x^2+x+1$. Then $U(p)$ is divisible by 5 if and only if $p \equiv \pm 1 \bmod 26$ or $p \equiv \pm 5 \bmod 26$. 
\end{ex}

\begin{ex}
    Let $P(x) = x^{10} + x^8 + x^7 + x^5 + x^3 + x^2 + 1$. Then $U(p)$ is not in general divisible by 2, 3 or 5.
\end{ex}

Since the number of factors of $U(n)$ increases as the number of distinct prime factors of $n$ increases, it might still be possible to prove that $U(n)$ is divisible by small primes when $n$ has sufficiently many distinct prime factors.

\begin{thm} \label{Undiv2}
    Let $P \in \mathbb{Z}[x]$ be a monic polynomial with constant coefficient $a_0 = \pm 1$ and let $k$ and $n$ be integers. For each root $\alpha_j$ of $P$, let $m_j$ be the order of $\alpha_j$ in the group $(R/kR)^\times$. If $m_j \mid \phi(n)$ for any $j$, then $k \mid U(n)$.
\end{thm}
\begin{proof}
    Let $j$ be such that $m_j \mid \phi(n)$. Then
    \begin{displaymath}
        \prod_{m \in \mathcal{D}_+(n)} \alpha_j^m - \prod_{m \in \mathcal{D}_-(n)} \alpha_j^m = \prod_{m \in \mathcal{D}_-(n)} \alpha_j^m (\alpha^{\phi(n)}-1) \equiv 0 \mod k,
    \end{displaymath}
    so $U(n)$ is divisible by $k$.
\end{proof}

\begin{cor}
    Let $P \in \mathbb{Z}[x]$ be a monic polynomial with constant coefficient $a_0 = \pm 1$. If $P(1) \neq 0$, then for all $n$, $U(n)$ is divisible by $P(1)$.
\end{cor}
\begin{proof}
    This follows by taking $k=P(1)$ in Theorem \ref{Undiv2}.
\end{proof}
\begin{rmk}
    If $P(1) = 0$, then $U(n) = 0$ for all $n$.
\end{rmk}

\bibliographystyle{plain}
\bibliography{references.bib}

\end{document}